\documentclass[12pt]{amsart}
\usepackage{amssymb,amsmath,amsfonts,latexsym, xcolor}
\usepackage{bm}

\newtheorem{propn}{Proposition}[section]
\newtheorem{thm}[propn]{Theorem}
\newtheorem{lemma}[propn]{Lemma}

\newtheorem{cor}[propn]{Corollary}

\newtheorem*{thm*}{Theorem}
\theoremstyle{definition}
\newtheorem{defn}[propn]{Definition}

\newtheorem{rem}{Remark}[section]

\newcommand{\Nat}{\mathbb{N}}

\newcommand{\Int}{\mathbb{Z}}
 \newcommand{\D}{\mathbb{D}}

 \DeclareMathOperator{\ran}{ran}

\newcommand{\clb}{\mathcal{B}}

\newcommand{\cle}{\mathcal{E}}
\newcommand{\clf}{\mathcal{F}}
\newcommand{\clg}{\mathcal{G}}
\newcommand{\clh}{\mathcal{H}}

\newcommand{\cls}{\mathcal{S}}

\newcommand{\clw}{\mathcal{W}}

\newcommand{\z}{\bm{z}}

\newcommand{\raro}{\rightarrow}

\begin{document}
\title[On product and partial isometry of Toeplitz operators]{On products and partial isometry of Toeplitz operators with operator-valued symbols}

\dedicatory{Dedicated to the memory of Dilip Kumar Guha, \\a loving and joyful grandfather.}

\author[Sarkar]{Srijan Sarkar}
\address{Department of Mathematics, Indian Institute of Technology Palakkad, Kerala - 678623, India. }
\email{srijans@iitpkd.ac.in,
srijansarkar@gmail.com}

\begin{abstract}
We solve the following problems associated with Toeplitz operators $T_{\Phi}$ on Hilbert space-valued Hardy spaces $H_{\mathcal{E}}^2(\mathbb{D}^n)$ over the unit polydisc $\D^n$.  \\
$(i)$~Given operator-valued bounded analytic functions $\Gamma, \Psi$ on $\mathbb{D}^n$, we completely characterize when the product $M_{\Gamma}M_{\Psi}^*$ becomes a Toeplitz operator by identifying tractable conditions on the functions. Furthermore, these conditions can be used to explicitly write the product into a sum of simple Toeplitz operators.\\
$(ii)$ We prove that partially isometric Toeplitz operators admit the following factorization:
\[
T_{\Phi} = M_{\Gamma} M_{\Psi}^*,
\]
where,  $\Gamma, \Psi$ are operator-valued inner functions on $\D^n$.  \\
A few of the immediate consequences are:\\
$(a)$~every partially isometric Toeplitz operator has a partially isometric symbol almost everywhere on $\mathbb{T}^n$ (distinguished boundary of $\mathbb{D}^n$). \\
$(b)$~any partially isometric analytic Toeplitz operator is of the form $M_{\Gamma V^*}$, where $\Gamma$ is an operator-valued inner function and $V$ is an constant isometry.\\
In connection with the result $(ii)$, we establish and use a crucial phenomenon: the range of partially isometric Toeplitz operators is always a Beurling-type invariant subspace of $H_{\mathcal{E}}^2(\mathbb{D}^n)$.  
Our results are new even in the case of Hardy spaces over the unit disc and extend the work of Brown--Douglas, Deepak--Pradhan--Sarkar on scalar-valued spaces. 
\end{abstract}



\maketitle

\section{Introduction and main results}
From its inception, Toeplitz operators have played a vital role in the interplay between operator theory and function theory. This connection has been instrumental in finding new results across many disciplines, like several complex variables, non-commutative geometry, mathematical physics, and engineering sciences. Recently, there has been an active interest in extending the results of Toeplitz operators on scalar-valued Hardy spaces to vector-valued spaces. The primary reason has always been to find a deeper understanding between the operators and their corresponding symbols. Several authors have contributed to showing that analytic Toeplitz operators with operator-valued symbols have many applications in systems engineering and $H^{\infty}$ control theory; for instance, see the celebrated monographs \cite{BRG, FF}. Let us now briefly discuss the setting.

For any Hilbert space $\cle$, the $\cle$-valued (or vector-valued) Hardy space on the unit polydisc $\D^n$ in the $n$-dimensional complex plane is defined by
\[
H_{\cle}^2(\D^n) := \big \{\sum_{\bm{k} \in \Nat^n} a_{\bm{k}} z^{\bm{k}} \in \mathcal{O}(\D^n,\cle): a_{\bm{k}} \in \cle,  \sum_{\bm{k} \in \Nat^n} \|a_{\bm{k}}\|_{\cle}^2 < \infty \big \}.
\]
In the case of $\cle = \mathbb{C}$, the above collection contains simply the scalar-valued functions. The space of all bounded operators on any Hilbert space (say) $\cle$ is denoted by $\clb(\cle)$, and the collection of operator-valued bounded analytic functions on $\D^n$ is denoted by $H_{\clb(\cle)}^{\infty}(\D^n)$. It is well-known that $H_{\cle}^2(\D^n)$ is isometrically isomorphic to the Hardy space $H_{\cle}^2(\mathbb{T}^n)$ on the $n$-torus $\mathbb{T}^n$ (contained inside the $n$-dimensional complex plane), and this characterization is achieved using radial limits \cite{NF}. Using this identification, a Toeplitz operator is defined in the following manner \cite{Douglas}.
\begin{defn}
A bounded linear operator $T$ on $H_{\cle}^2(\D^n)$ is said to be Toeplitz if there exists an operator-valued function $\Phi \in L_{\clb(\cle)}^{\infty}(\mathbb{T}^n)$ such that $T = P_{H_{\cle}^2(\D^n)} L_{\Phi}|_{H_{\cle}^2(\D^n)},$ where $L_{\Phi}$ is the Laurent operator on $L_{\cle}^2(\mathbb{T}^n)$ (that is, the multiplication operator) associated to $\Phi.$ In this case, $T$ is denoted by $T_{\Phi}$, and the function $\Phi$ is called the symbol of the operator $T_{\Phi}.$
\end{defn}
It is immediate that when $\Phi$ is an operator-valued bounded analytic function on $\D^n$, the corresponding Toeplitz operator is simply the multiplication operator on $H_{\cle}^2(\D^n)$. For this reason, we make a distinction and use the following convention throughout this article.
\[
\text{If } \Phi \in H_{\clb(\cle)}^{\infty}(\D^n), \text{ then the Toeplitz operator is denoted by } M_{\Phi}.
\]
It is worth mentioning that unlike in the case of $H^2(\D)$, the theory of Toeplitz operators on both $H^2(\D^n)$ and $H_{\cle}^2(\D^n)$ (where $n>1$) is far from being complete, as many natural questions remain unanswered. Several challenges appear whenever we leave the realm of scalar-valued functions and delve into the world of operator-valued symbols.  We refer the reader to the exceptional monograph by Douglas for results on matrix-valued Toeplitz operators \cite{Douglas}. Recently, Curto, Hwang, and Lee have made some spectacular progress in this direction by studying Halmos's question of subnormal Toeplitz operators for matrix or operator-valued symbols \cite{CHL, CHL2, CHL3}.  

In this article, we further explore this theme: the correspondence of Toeplitz operators with their symbols under operator-theoretic conditions.  From the algebraic characterization of Toeplitz operators by Brown and Halmos \cite{BH}, it can be realized that the only possibility for any $T_{\phi} \in \clb(H^2(\D))$ to become an isometry (that is, $\|T_{\phi}f\| = \|f\|$ for all $f \in H^2(\D)$) is when its symbol $\phi$ is an inner function (that is, $\phi \in H^{\infty}(\D)$ and $|\phi(e^{it})|=1$ almost everywhere on $\mathbb{T}$). Brown and Douglas were interested in the behaviour of Toeplitz operators under the general condition of it being a partial isometry. We recall that a bounded operator $T$ on a Hilbert space $\clh$ is a partial isometry if $T$ is an isometry on the orthogonal complement of its kernel. In \cite{BD}, the authors proved that the only partially isometric Toeplitz operators on $H^2(\D)$ are of the form $T_{\phi} = M_{\theta}$, or else, $T_{\phi} = M_{\theta}^*$, where $\theta$ is an inner function on $\D$. Recently, Deepak--Pradhan--Sarkar extended this result for Toeplitz operators on $H^2(\D^n)$, where $n>1$.

\begin{thm*}{\cite[Theorem 1.1]{KPS}}
Let $\phi$ be a non-zero function in $L^{\infty}(\mathbb{T}^n)$. Then $T_{\phi}$ is a partial isometry if and only if there exist inner functions $\phi_1, \phi_2 \in H^{\infty}(\D^n)$ such that $\phi_1$ and $\phi_2$ depends on different variables and $T_{\phi} = M_{\phi_1}^* M_{\phi_2}$. 
\end{thm*}
In other words, the condition of partial isometry on $T_{\phi}$ forces a factorization $\phi = \bar{\phi_1} \phi_2$ almost everywhere (denoted by a.e.) on $\mathbb{T}^n$. An immediate question is finding the corresponding result for Toeplitz operators with operator-valued symbols, hoping that a satisfying answer will explain the reasoning behind the surprising factorization of the symbol $\phi$. If we expect a straightforward generalization, then there lies an immediate obstacle. In scalar-valued cases, both Brown--Douglas and Deepak--Pradhan--Sarkar's results show that if we start with $\phi \in H^{\infty}(\D^n)$, then the only possibility for the analytic Toeplitz operator $M_{\phi}$ to be a partial isometry is when $M_{\phi}$ is an isometry, in other words, $\phi$ must be an inner function. Many examples involving block Toeplitz operators show that this no longer holds in the setting of operator-valued symbols. First, let us note that an operator-valued function $\Theta \in H_{\clb(\clf,\cle)}^{\infty}(\D^n)$ is said to be an inner function if $\Theta(\bm{z}): \clf \raro \cle$ is an isometry a.e. on $\mathbb{T}^n$ \cite{NF}. Let us now consider the following bounded analytic matrix-valued symbol,
\[
\Theta(z) := \begin{bmatrix}
0 &z\\
0 &0
\end{bmatrix} \in H_{\clb(\mathbb{C}^2)}^{\infty}(\D),
\]
then for each $\lambda \in \mathbb{T}$,
\[
\Theta(\lambda)^* \Theta(\lambda) = \begin{bmatrix}
0 &0\\
\bar{\lambda} &0
\end{bmatrix}  \begin{bmatrix}
0 &\lambda\\
0 &0
\end{bmatrix} = \begin{bmatrix}
0 &0\\
0 &1
\end{bmatrix}.
\]
This implies that $\Theta(z)$ is not an inner function. However, 
\[
M_{\Theta}^* M_{\Theta} = T_{\Theta^* \Theta} = \begin{bmatrix}
0 &0\\
0 & I_{H^2(\D)}
\end{bmatrix} \in H_{\mathbb{C}^2}^2(\D),
\]
shows that $M_{\Theta}$ is indeed a partial isometry. Thus, a non-constant analytic Toeplitz operator $M_{\Theta}$ can be a partial isometry without $\Theta(z)$ being an inner function.  Motivated by this observation and earlier results, we study the following intriguing question in this article: \textit{What are the partially isometric Toeplitz operators on vector-valued Hardy spaces?}

Now, let us digress a little to highlight two important directions associated with this question. The first is in the study of invariant subspaces of $M_z \oplus M_z^*$ on $H_{\cle}^2(\D) \oplus H_{\clf}^2(\D)$ as observed by Gu and Luo in \cite{GL}. This work itself was motivated by the characterization of invariant subspaces of $M_z \oplus M_z^*$ on $H^2(\D) \oplus H^2(\D)$ by Timotin \cite{Timotin}. In their work, Gu and Luo observed that the invariant subspace of $M_z \oplus M_z^*$ is related to the range of the following operator.
\[
V_{\Phi} = \begin{bmatrix}
T_A & T_B\\
H_C & H_D
\end{bmatrix}: H_{\cle}^2(\D) \oplus H_{\clf}^2(\D) \raro H_{\cle}^2(\D) \oplus H_{\clf}^2(\D),
\]
where 
\[
\Phi(z) = \begin{bmatrix}
A(z) & B(z)\\ 
C(z) & D(z)
\end{bmatrix}: \cle \oplus \clf  \raro \cle \oplus \clf \quad (z \in \D),
\]
and $T_A, T_B$, and $H_C, H_D$ are the corresponding Toeplitz and Hankel operators, respectively.  Furthermore, the authors found that the range of $V_{\Phi}$ is closed in $H_{\cle}^2(\D) \oplus H_{\clf}^2(\D)$ if and only if $V_{\Phi}$ is a partial isometry. This led to the natural question of when $V_{\Phi}$ becomes a partial isometry. In this regard, they made the following statement.

``The above problem seems a difficult one since the partial isometric characterizations
of Toeplitz operator $T_A$ (A is not necessarily analytic) and Hankel operator $H_D$ are only known when $A$ and $D$ are scalar functions in $L^{\infty}$ \cite{BD, Peller}." In this article, we completely resolve this issue for Toeplitz operators; see Theorem \ref{mainthm1} and Theorem \ref{main_analytic}. 

The second connection lies in studying Beurling-type invariant subspaces of Hardy spaces over the unit polydisc. In a recent article \cite{DPS},  Debnath et al. used the characterization \cite[Theorem 1.1]{KPS} to give a complete answer to the following question: given non-constant inner functions $\theta_1, \theta_2$ on $\D^n$, when does the orthogonal projections onto invariant subspaces $\theta_1 H^2(\D^n)$ and $\theta_2 H^2(\D^n)$ commute with each other? This resulted in a definite answer to the following question raised by R.G. Douglas:  given non-constant inner functions $\theta_1, \theta_2$ on $\D^n$, when does the product of projections onto the orthogonal subspaces $(\theta_1 H^2(\D^n))^{\perp}$ and $(\theta_2 H^2(\D^n))^{\perp}$ become a finite-rank projection?

Thus, it is evident that our main result will serve as an essential component for\\
$(i)$~the characterization of invariant subspaces of sums of shift operators on vector-valued Hardy spaces over both domains $\D$ and $\D^n$;\\
$(ii)$~answering Douglas's question in the setting of vector-valued Hardy spaces by studying commuting projections onto Beurling-type shift-invariant subspaces of $H_{\cle}^2(\D^n)$.

Since the proofs in \cite{BH, KPS} rely crucially on the commutativity of the symbols, the methods used in these works are not at all applicable for answering the above questions. Thus, it is essential to find a completely new approach to the above question.  Based on earlier works, we can hope that a partially isometric $T_{\Phi}$ may admit a factorization into the product of Toeplitz operators corresponding to inner symbols.  For this purpose, let us briefly elucidate how multiplication operators corresponding to inner functions appear in the context of Hardy spaces. Recall that a closed $(M_{z_1},\ldots, M_{z_n})$-joint invariant subspace of $H_{\cle}^2(\D^n)$ is said to be \textit{Beurling-type} if there exists a Hilbert space $\clf$, and an inner function $\Theta \in H_{\clb(\clf,\cle)}^{\infty}(\D^n)$ such that $\cls = M_{\Theta} H_{\clf}^2(\D^n)$. It is well-known that given a subspace $\cls \subseteq H_{\cle}^2(\D^n)$ if we consider the restriction operators $R_i: = M_{z_i}|_{\cls}$ for all $i \in \{1,\ldots,n\}$, then $\cls$ is Beurling-type invariant subspace if and only if the cross-commutators $[R_i^*, R_j] = 0$ for all distinct $i, j \in \{1,\ldots,n\}$ (\cite{Mandrekar, SSW}). Thus, if we can show that the range of Toeplitz operators satisfies the above algebraic conditions concerning the restriction operators,  then we can associate inner functions to the corresponding symbol. In section \ref{Beurling}, we establish this connection by proving the following result.
\begin{thm}\label{doubcom}
If $T_{\Phi}$ is a non-zero partially isometric Toeplitz operator on $H_{\cle}^2(\D^n)$, then $\ran T_{\Phi}$ is a Beurling-type invariant subspace of $H_{\cle}^2(\D^n)$.
\end{thm}
Let us note that classically, the kernel of the adjoint of Toeplitz operators on $H^2(\D)$ is related to \textit{nearly} invariant subspaces of $H^2(\D)$. This observation was noted by Hayashi in \cite{Hayashi} and further explored by Sarason in \cite{Sarason}.  We believe that the above result should serve as an impetus to study the range of Toeplitz operators as well.  

With the above characterization in mind, the author could anticipate that a partially isometric Toeplitz operator $T_{\Phi}$ may admit the following factorization.
\[
T_{\Phi} = M_{\Gamma} M_{\Psi}^*,
\]
where $\Gamma, \Psi$ are operator-valued inner functions on $\D^n$, but this led to another challenge: \textit{under what conditions does $M_{\Gamma} M_{\Psi}^*$ become a Toeplitz operator?}

On the scalar-valued Hardy spaces, it is well-known when the product of two Toeplitz operators is again a Toeplitz operator. Brown and Halmos developed this result for $H^2(\D)$ in  \cite{BH}, and Gu proved the corresponding result for $H^2(\D^n)$ in \cite{Gu2}. However, the answer to this question for Toeplitz operators with operator-valued symbols is still unclear.  We refer the reader towards some important progress in the case of block Toeplitz operators by Gu and Zheng in \cite{GZ, Gu}.  This is where we make another significant contribution. Section \ref{product} is dedicated to establishing tractable conditions that solve the above product problem for not only inner but also bounded analytic functions.  Before stating the result, let us set the following notation: for any $k \in \{1,\ldots,n\}$, the functions $\Gamma_k(\bm{z}), \Psi_k(\bm{z})$ on $\D^n$ denote the functions $\Gamma$ and $\Psi$, respectively, with $0$ in the $k$-th coordinate.

\begin{thm}\label{prod_mains}
Let $\clf, \cle$ be Hilbert spaces and $\Gamma(\bm{z}), \Psi(\bm{z})$ be $\clb(\clf,\cle)$-valued bounded analytic functions on $\D^n$. Then $M_{\Gamma}M_{\Psi}^*$ is a Toeplitz operator on $H_{\cle}^2(\D^n)$ if and only if 
\[
\big( \Gamma(\bm{\lambda}) - \Gamma_k(\bm{\lambda}) \big) \big( \Psi(\bm{\mu}) - \Psi_k(\bm{\mu})\big)^* = 0,
\]
for all $\bm{\lambda}, \bm{\mu} \in \D^n$, and $k \in \{1,\ldots,n\}$. 

\end{thm}
These conditions further arise from an equivalent set of necessary and sufficient conditions on the Fourier coefficients of the symbols $\Gamma$ and $\Psi$ (see Proposition \ref{main2}). Surprisingly, these conditions also serve as an algorithm to write the product $M_{\Gamma} M_{\Psi}^*$ into sums of elementary Toeplitz operators,  as can be seen in the following case when $n=1$.
\begin{thm}\label{simple}
Let $\clf, \cle$ be Hilbert spaces and $\Gamma(z), \Psi(z)$ be $\clb(\clf,\cle)$-valued bounded analytic functions on $\D$. Then the following are equivalent.\\ \vspace{1mm}
$(i)$~$M_{\Gamma}M_{\Psi}^*$ is a Toeplitz operator on $H_{\cle}^2(\D)$,\\ \vspace{1mm}
$(ii)$~$\big( \Gamma(\lambda) - \Gamma(0) \big) \big( \Psi(\mu) - \Psi(0) \big)^* = 0$ for all $\lambda, \mu \in \D$,\\ \vspace{1mm}
$(iii)$~$M_{\Gamma}M_{\Psi}^*$ admits the following decomposition
\[
M_{\Gamma}M_{\Psi}^*  = M_{\Gamma} \Psi(0)^* + \Gamma(0) M_{\Psi}^* - \Gamma(0) \Psi(0)^*.
\]
\end{thm}
The terms on the right-hand side of the above identity are analytic, co-analytic and constant Toeplitz operators, clearly showing how $M_{\Gamma} M_{\Psi}^*$ becomes a Toeplitz operator. An immediate consequence is the following: \textsf{when $M_{\Gamma} M_{\Psi}^*$ is a Toeplitz operator, $\Gamma(0)=0$ or $\Psi(0)=0$, implies that $M_{\Gamma} M_{\Psi}^*$ is analytic or co-analytic, respectively}.  Also, see Remark \ref{interes} for a function-theoretic consequence of the above characterization.

The decomposition of $M_{\Gamma}M_{\Psi}^*$ into the sum of elementary Toeplitz operators are challenging to write down for $n>1$ cases, and the difficulty increases with increasing $n$.  In this article, we have meticulously written down the decompositions in the cases of $n=2,3$ (Theorem \ref{bidisc} and Theorem \ref{tridisc}, respectively) and also highlighted in Remark \ref{orem} the method to find the decomposition when $n>3$.  In section \ref{partial}, we use Theorem \ref{prod_mains} to give a complete characterization for partially isometric Toeplitz operators on $H_{\cle}^2(\D^n)$.
\begin{thm}\label{mainthm1}
A non-constant Toeplitz operator $T_{\Phi}$ on $H_{\cle}^2(\D^n)$ is a partial isometry if and only if there exists a Hilbert space $\clf$,  and inner functions $\Gamma(\bm{z}), \Psi(\bm{z}) \in H_{\clb(\clf, \cle)}^{\infty}(\D^n)$ such that 
\[
T_{\Phi} = M_{\Gamma} M_{\Psi}^*,
\]
where $\Gamma(\bm{z})$ and $\Psi(\bm{z})$ satisfy
\[
\big( \Gamma(\bm{\lambda}) - \Gamma_k(\bm{\lambda}) \big) \big( \Psi(\bm{\mu}) - \Psi_k(\bm{\mu})\big)^* = 0,
\]
for all $\bm{\lambda}, \bm{\mu} \in \D^n$, and $k \in \{1,\ldots,n\}$
\end{thm}
For elaborate statements in the cases of $n=1,2,3$, we refer to Theorem \ref{psimple}, Theorem \ref{pbidisc},  and Theorem \ref{ptridisc}, respectively. As expected, there are numerous consequences of this result. For instance, a partially isometric $T_{\Phi}$ will always have a partially isometric symbol, a.e., on $\mathbb{T}^n$ (Corollary \ref{symb}). At the end of section \ref{partial}, we use the above result to characterize partially isometric Toeplitz operators that have analytic symbols (Theorem \ref{main_analytic}), that are hyponormal (Theorem \ref{hyponormal}), and normal (Corollary \ref{normal}).  In section \ref{abstract}, we use our main result to characterize general Toeplitz operators on abstract Hilbert spaces.

In section \ref{scalar}, we give a completely new proof of the main results obtained by Deepak et al. in \cite[Theorem 1.1]{KPS} and Brown et al in \cite{BD}. \textsf{This new approach shows that, in general, the condition of partial isometry on a Toeplitz operator $T_{\Phi}$ induces a factorization of the symbol $\Phi$ into the product $\Gamma \Psi^*$ where $\Gamma$ and $\Psi$ are inner functions jointly satisfying certain function-theoretic conditions. In scalar cases, these conditions manifest into the factorization of $\phi$ into the product of functions depending on disjoint variables (see Theorem \ref{comcondn}).}

Let us now describe the plan for the rest of this article. In section \ref{prelim}, we set notations, definitions, and establish a few results essential for the later part. We end this article with section \ref{questions}, where we have highlighted several interesting questions worthy of further investigation.

\section{Preliminaries}\label{prelim}
In this section, we set the notations, definitions, and results needed in this article.  Let us begin by looking at the following equivalent way of defining $\cle$-valued Hardy spaces on $\D^n$.
{\small 
\[
H_{\cle}^2(\D^n): = \{f \in \mathcal{O}(\D^n, \cle): \|f\|_2^2: = \sup_{0<r<1} \int_{\mathbb{T}^n} \|f(re^{i \theta_1}, \ldots, re^{i \theta_n})\|_{\cle}^2 d\mu < \infty\},
\]
}
where $\mu$ is the normalized Lebesgue measure on $\mathbb{T}^n$. This space of $\cle$-valued analytic functions has a natural collection of shift operators, namely,
\[
M_{z_i} f := z_i f, \quad (f \in H_{\cle}^2(\D^n)),
\]
for each $i \in \{1,\ldots,n\}$. We will denote by $M_{\bm{z}} = (M_{z_1},\ldots, M_{z_n})$ as the tuple of shift operators. For $\bm{k} = (k_1,\ldots,k_n) \in \Nat^n$, we set the following convention,
\[
M_{\bm{z}}^{\bm{k}} = M_{z_1}^{k_1} \cdots M_{z_n}^{k_n},
\]
and for any $\bm{k} = (k_1,\ldots,k_n), \bm{l} = (l_1, \ldots, l_n) \in \Nat^n$, we say $\bm{k} \leq  \bm{l}$ if $k_i \leq l_i$ for all $i \in \{1,\ldots,n\}$. For the tuple of shift operators $L_{e^{i\theta}} = (L_{e^{i\theta_1}}, \ldots, L_{e^{i\theta_n}})$ on  $L^2(\mathbb{T}^n)$, we set
\[
L_{e^{i\theta}}^{\bm{k}} = L_{e^{i \theta_1}}^{k_1} \cdots L_{e^{i \theta_n}}^{k_n}.
\]
Brown and Halmos gave an algebraic characterization for Toeplitz operators on $H^2(\D)$ \cite{BH}. It is well-known that there exists a natural extension of this result to Toeplitz operators on $H^2(\D^n)$ \cite{MSS}. The following result shows that such a characterization holds for $H_{\cle}^2(\D^n)$ as well. The result can be proved verbatim from \cite[Theorem 3.1, Theorem 5.2]{MSS}. For the sake of completeness, we give a sketch of the proof. 
\begin{thm}\label{Toep_char}
A bounded operator $T$ on $H_{\cle}^2(\D^n)$ is a Toeplitz operator if and only if 
\[
M_{z_i}^* T M_{z_i} = T,
\]
for all $i \in \{1,\ldots,n\}$. 
\end{thm}
\begin{proof}
For each $k \in \Nat$, consider $\bm{k}_d := (k,\ldots,k) \in \Nat^n$.  From the assumption $M_{z_i}^* T M_{z_i} = T$, we immediately get
\[
M_{\bm{z}}^{* \bm{k}_d} T M_{\bm{z}}^{ \bm{k}_d} = T.
\]
This implies that for any $\eta, \zeta \in \cle$,
\[
\langle T e_{\bm{i} + \bm{k}_d} \eta, e_{\bm{i} + \bm{k}_d} \zeta\rangle = \langle T M_{\bm{z}}^{\bm{k}_d}e_{\bm{i}} \eta,  M_{\bm{z}}^{\bm{k}_d} e_{\bm{j}} \zeta\rangle  = \langle T e_{\bm{i}}, e_{\bm{j}} \rangle,
\]
for all $\bm{i}, \bm{j} \in \Nat^n$, and $k \in \Nat$.  Now, for each $\bm{l}, \bm{m} \in \Int^n$ there exists $\bm{t} = (t_1,\ldots,t_n) \in \Nat^n$ such that $\bm{l}+\bm{k}_d, \bm{m}+\bm{k}_d \in \Nat^n$ for all $\bm{k}_d \geq \bm{t}$. Using this observation, if we set $A_k := L_{e^{i\theta}}^{* \bm{k}_d} T P_{H_{\cle}^2(\D^n)} L_{e^{i\theta}}^{\bm{k}_d}$ for $k \in \Nat \setminus \{0\}$, then we get
\[
\langle A_k e_{\bm{l}}, e_{\bm{m}} \rangle_{L_{\cle}^2(\mathbb{T}^n)} = \langle T P_{H_{\cle}^2(\D^n)} e_{\bm{l}+\bm{k}_d}, e_{\bm{m}+\bm{k}_d} \rangle_{L_{\cle}^2(\mathbb{T}^n)},
\]
and therefore, for all $\bm{k}_d \geq \bm{t}$
\[
\langle A_k e_{\bm{l}}, e_{\bm{m}} \rangle \raro \langle T e_{\bm{l}+t}, e_{\bm{m}+t} \rangle \text{ as } k \raro \infty.
\]
Hence, we can define a bounded bilinear form $\textsf{b}(\cdot,\cdot)$ on the linear span of $\{e_{\bm{s}} \zeta: \bm{s} \in \Int^n, \zeta \in \cle\}$, in the following manner,
\[
\textsf{b}(e_{\bm{l}}\eta, e_{\bm{m}}\zeta) = \lim_{k \raro \infty} \langle A_k e_{\bm{l}}\eta, e_{\bm{m}}\zeta \rangle,
\]
for all $\bm{l}, \bm{m} \in \Int^n$ and $\eta, \zeta \in \cle$. Therefore, there exists an operator $A_{\infty} \in \clb(L_{\cle}^2(\mathbb{T}^n))$ such that 
\[
\langle A_{\infty} f,g \rangle = \lim_{k \raro \infty} \langle A_k f,g \rangle,
\]
for $f,g \in L_{\cle}^2(\mathbb{T}^n)$. Let $\epsilon_j = (0,\ldots, 0,\underset{\text{j-th position}}{1},0, \ldots,0)$. Then for all $k$ sufficiently large (depending on $\bm{l}, \bm{m}$ and $j$), we get
\begin{align*}
&\langle L_{e^{i \theta}}^{* \bm{k}_d} T P_{H_{\cle}^2(\D^n)} L_{e^{i \theta}}^{\bm{k}_d} e_{\bm{l} + \epsilon_j}, e_{\bm{m} + \epsilon_j} \rangle_{L_{\cle}^2(\mathbb{T}^n)}  \\
&= \langle T P_{H_{\cle}^2(\D^n)}  e_{\bm{l} + \bm{k}_d + \epsilon_j}, e_{\bm{m} + \bm{k}_d + \epsilon_j} \rangle_{L_{\cle}^2(\mathbb{T}^n)} \\
&= \langle A_{k} e_{\bm{l}}, e_{\bm{m}} \rangle_{L_{\cle}^2(\mathbb{T}^n)}.
\end{align*}
This immediately gives
\begin{align*}
\langle A_{\infty} e_{\bm{l} + \epsilon_j}, e_{\bm{m} + \epsilon_j} \rangle_{L_{\cle}^2(\mathbb{T}^n)} &= \lim_{k \raro \infty} \langle L_{e^{i \theta}}^{* \bm{k}_d} T P_{H_{\cle}^2(\D^n)} L_{e^{i \theta}}^{\bm{k}_d} e_{\bm{l} + \epsilon_j}, e_{\bm{m} + \epsilon_j} \rangle_{L_{\cle}^2(\mathbb{T}^n)} \\
&= \langle A_{\infty} e_{\bm{l}}, e_{\bm{m} } \rangle_{L_{\cle}^2(\mathbb{T}^n)}.
\end{align*}
This implies that
\[
A_{\infty}L_{e^{i \theta_j}} = L_{e^{i \theta_j}} A_{\infty},
\]
for all $j=1,\ldots,n$. Hence, there exists $\Phi(z) \in L_{\clb(\cle)}^{\infty}(\mathbb{T}^n)$ such that
\[
A_{\infty} = L_{\Phi},
\]
which further implies that $T = P_{H_{\cle}^2(\D^n)} L_{\Phi}|_{H_{\cle}^2(\D^n)}$. 

Conversely, if we begin with $T = P_{H_{\cle}^2(\D^n)} L_{\Phi}|_{H_{\cle}^2(\D^n)}$ for some $\Phi(z) \in L_{\clb(\cle)}^{\infty}(\mathbb{T}^n)$, then for any $f,g \in H_{\cle}^2(\D^n)$ and $j=1,\ldots,n$ we get
\[
\langle M_{z_j}^* TM_{z_j}f,g \rangle_{H_{\cle}^2(\D^n)} = \langle \Phi e^{i \theta_j}f, e^{i \theta_j}g \rangle_{L_{\cle}^2(\mathbb{T}^n)} = \langle \Phi f, g \rangle_{L_{\cle}^2(\mathbb{T}^n)}.
\]
In other words, 
\[
\langle M_{z_j}^* TM_{z_j}f,g \rangle_{H_{\cle}^2(\D^n)} = \langle P_{H_{\cle}^2(\D^n)} L_{\Phi}|_{H_{\cle}^2(\D^n)}f,g \rangle_{H_{\cle}^2(\D^n)},
\]
which implies that $M_{z_j}^* T M_{z_j} = T$ for all $j=1,\ldots,n$. This completes the proof.
\end{proof}
This property plays an important role in developing many results on Toeplitz operators and will be used throughout this article.  Before moving into some of its consequences, let us briefly recall some concepts on partial isometries.
\begin{defn}
$T \in \clb(\clh)$ is a partial isometry if $T$ is an isometry on the orthogonal complement of the subspace $\{h \in \clh: \|Th\|=0\}$ (also referred to as kernel of $T$).
\end{defn}
It is well-known that $T$ is a partial isometry on $\clh$ if and only if $T$ satisfies any one of the following equivalent conditions.\\
$(i)$~$TT^*$ is a projection.\\
$(ii)$~$TT^*T = T$.\\
From this description, the following useful facts can be deduced.\\
$(i)$~$T$ is a partial isometry if and only if $T^*$ is also a partial isometry.\\
$(ii)$~The range of a partially isometric operator must be a closed subspace of $\clh$.\\
Most importantly, in this case, we have the following identity
\begin{equation}\label{pi}
I_{\clh} - TT^* = P_{\ker T^*}.
\end{equation}
where $P_{\ker T^*}$ is the orthogonal projection onto the kernel of $T^*$. Throughout this article, we shall use this identity, especially in the case of shift operators. We refer the readers to some fundamental results on partial isometries by Halmos, McLaughlin, and Wallen in \cite{HM, HW}.

\begin{thm}
Let $T_{\Phi}$ be a Toeplitz operator on $H_{\cle}^2(\D^n)$. Then $T_{\Phi}$ is an isometry if and only if $\Phi$ is an inner function in $H_{\clb(\cle)}^{\infty}(\D^n)$.
\end{thm}
\begin{proof}
Any isometry $V$ on a Hilbert space $\clh$ satisfies the identity $V^*V = I_{\clh}$.  Using this property we observe that if $T_{\Phi}$ is an isometry, then
\[
I_{H_{\cle}^2(\D^n)} = T_{\Phi}^* T_{\Phi} = M_{z_i}^* T_{\Phi}^* M_{z_i}M_{z_i}^* T_{\Phi} M_{z_i} \quad (\forall i \in \{1,\ldots,n\}).
\]
Since $M_{z_i}$ is an isometry for each $i \in \{1,\ldots,n\}$, we can use identity (\ref{pi}) to get the following.
\[
M_{z_i}^* T_{\Phi}^*  T_{\Phi} M_{z_i} - M_{z_i}^* T_{\Phi}^*  P_{\ker M_{z_i}^*}T_{\Phi} M_{z_i}= I_{H_{\cle}^2(\D^n)} - M_{z_i}^* T_{\Phi}^*  P_{\ker M_{z_i}^*}T_{\Phi} M_{z_i},
\]
for all $i \in \{1,\ldots,n\}$. Thus, we get that $M_{z_i}^* T_{\Phi}^*  P_{\ker M_{z_i}^*}T_{\Phi} M_{z_i}=0$, which further implies,
\[
P_{\ker M_{z_i}^*}T_{\Phi} M_{z_i} =0.
\]
Now, $P_{\ker M_{z_i}^*}T_{\Phi} M_{z_i} = (I_{H^2(\D^n)} - M_{z_i}M_{z_i}^*)T_{\Phi} M_{z_i} = T_{\Phi} M_{z_i}  -M_{z_i}  T_{\Phi}$ implies that the above identity should force
$M_{z_i}  T_{\Phi} = T_{\Phi} M_{z_i}$ for all $i \in \{1,\ldots,n\}$ and hence, $\Phi$ must be in $H_{\clb(\cle)}^{\infty}(\D^n)$ \cite{NF}. Since $T_{\Phi}$ is an isometry, $\Phi$ must be an inner function in $H_{\clb(\cle)}^{\infty}(\D^n)$. For the converse direction, it is well-known that a Toeplitz operator corresponding to an inner function is always an isometry \cite{NF}. This completes the proof.
\end{proof}
\begin{cor}\label{unitary}
Let $T_{\Phi}$ be a Toeplitz operator on $H_{\cle}^2(\D^n)$, then $T_{\Phi}$ is a unitary if and only if $\Phi$ is a constant unitary on $\cle$.
\end{cor}
\begin{proof}
From the assumption, both $T_{\Phi}$ and $T_{\Phi^*} = T_{\Phi}^*$ are isometries. Following the above result, both $\Phi$, and $\Phi^*$ must be inner functions in $H_{\clb(\cle)}^{\infty}(\D^n)$. Thus, $\Phi$ must be a constant operator on $\cle$. Since $T_{\Phi}$ is unitary, $\Phi$ must be a constant unitary on $\cle$. This completes the proof.
\end{proof}
Similar, to the case of Toeplitz operators on $H_{\cle}^2(\D)$ \cite{CHL}, we can consider the following decomposition on $\D^n$ for symbols $F,G \in L_{\clb(\cle)}^{\infty}(\mathbb{T}^n)$,
\[
T_{FG} = T_{F} T_{G} + H_{F^*}^* H_{G}.
\]
Here $H_{F}: = JP_{H_{\cle}^{2}(\D^n)^{\perp}} L_{F}|_{H_{\cle}^{2}(\D^n)}: H_{\cle}^{2}(\D^n) \raro H_{\cle}^{2}(\D^n)$ is the Hankel operator with symbol $\Phi$, and $J: (H_{\cle}^2(\D^n))^{\perp} \raro  H_{\cle}^2(\D^n)$ is the unitary defined by $J(f)(\bm{z}) = \bar{\bm{z}} I_{\cle}f(\bar{\bm{z}})$, for $f \in (H_{\cle}^2(\D^n))^{\perp}$. It is well known that Hankel operators satisfy the following property for each $i \in \{1,\ldots,n\}$. 
\begin{equation}\label{comm}
M_{z_i}^* H_F = H_F M_{z_i}.
\end{equation}
We refer the reader to the monographs \cite{BS, Peller} for an elaborate discussion of the above facts. We can establish the following well-known fact using the above intertwining property.

\begin{thm}\label{ptoep}
If $T_F T_G$ is a Toeplitz operator on $H_{\cle}^2(\D^n)$, then $T_F T_G = T_{FG}$.
\end{thm}
\begin{proof}
Note that $T_F T_G = T_{FG} - H_{F^*}^* H_G$. If $T_F T_G = T_Y$ is a Toeplitz operators (say) $T_Y$, then 
\[
T_Y =  T_{FG} - H_{F^*}^* H_G.
\]
This implies that $T_{Y - FG} =  H_{F^*}^* H_G$. But then we have
\[
M_{z_i}^{*k} H_{F^*}^* H_G M_{z_i}^{k} = H_{F^*}^* H_G.
\]
Using the intertwining identity for Hankel operators, we get
\[
M_{z_i}^{*k} H_{F^*}^* M_{z_i}^{*k} H_G  = H_{F^*}^* H_G,
\]
and therefore, for each $x \in H_{\cle}^2(\D^n)$, we must have
\[
\|H_{F^*}^* H_G x\| \leq \|M_{z_i}^{*k} H_G  x\| \underset{k \raro \infty}{\raro} 0.
\]
It is because for all $i \in \{1,\ldots,n\}$, the operator $M_{z_i}$ is a pure isometry, that is,  $M_{z_i}^{*k} \raro 0$ in the strong operator topology as $k \raro \infty$. The above inequality implies that $H_{F^*}^* H_G =0$, and hence, $T_F T_G = T_{FG} $. This completes the proof.
\end{proof}

\section{Product of Toeplitz operators}\label{product}
In this section, we aim to establish necessary and sufficient conditions for which a certain product of Toeplitz operators corresponding to analytic symbols is again a Toeplitz operator. Let us begin with a lemma useful for finding conditions in the general situation.

\begin{lemma}\label{toepcondn1}
Let $\clf, \cle$ be Hilbert spaces and $\Gamma \in H_{\clb(\clf, \cle)}^{\infty}(\D^n)$ and $\Psi \in H_{\clb(\clf,  \cle)}^{\infty}(\D^n)$ be operator-valued bounded analytic functions. Then the following are equivalent\\
$(i)$~$M_{\Gamma}M_{\Psi}^*$ is a Toeplitz operator on $H_{\cle}^2(\D^n)$.\\
$(ii)$~$M_{z_i}^*M_{\Gamma} P_{\ker M_{z_i^*}} M_{\Psi}^*  M_{z_i}= 0$, for all $i \in \{1,\ldots,n\}$.
\end{lemma}
\begin{proof}
Using identity (\ref{pi}), we know that 
\[
I_{H^2_{\cle}(\D^n)} = M_{z_i}M_{z_i}^* + P_{\ker M_{z_i}^*},
\]
for all $i \in \{1,\ldots,n\}$. Therefore, for any $i \in \{1,\ldots,n\}$ we get
\[
M_{z_i}^* M_{\Gamma} M_{\Psi}^* M_{z_i} =M_{\Gamma}M_{\Psi}^* + M_{z_i}^* M_{\Gamma} P_{\ker M_{z_i}^*} M_{\Psi}^* M_{z_i}.
\]
From Theorem \ref{Toep_char}, it is clear that $M_{\Gamma}M_{\Psi}^*$ is a Toeplitz operator if and only if $M_{z_i}^* M_{\Gamma} M_{\Psi}^* M_{z_i} = M_{\Gamma} M_{\Psi}^*$ for all $i \in \{1,\ldots,n\}$. In other words, $M_{z_i}^* M_{\Gamma} P_{\ker M_{z_i}^*} M_{\Psi}^* M_{z_i}=0$ for all $i \in \{1,\ldots,n\}$.  This completes the proof.
\end{proof}
We can further strengthen the above condition in the following manner.
\begin{lemma}\label{toepcondn2}
Let $\clf, \cle$ be Hilbert spaces and $\Gamma(\bm{z}) \in H_{\clb(\clf, \cle)}^{\infty}(\D^n)$ and $\Psi(\bm{z}) \in H_{\clb(\clf,  \cle)}^{\infty}(\D^n)$. Then the following are equivalent for any $i \in \{1,\ldots,n\}$.\\
$(i)$~$M_{z_i}^*M_{\Gamma} P_{\ker M_{z_i^*}} M_{\Psi}^*  M_{z_i}= 0$\\
$(ii)$~$M_{z_i}^*M_{\Gamma} P_{\cle} M_{\Psi}^*  M_{z_i}= 0$.
\end{lemma}
\begin{proof}
Let us begin with the direction $(i) \implies (ii)$. Note that for any $i \in \{1,\ldots,n\}$,  
\[
M_{z_i}^*M_{\Gamma} P_{\ker M_{z_i^*}} M_{\Psi}^*  M_{z_i}= 0,
\]
implies that
\[
M_{z_i}^* M_{\Gamma} \big( M_{z_j} M_{z_j}^* + P_{\ker M_{z_j}^*} \big) P_{\ker M_{z_i}^*}  M_{\Psi}^* M_{z_i} = 0,
\]
for any $j \in \{1,\ldots,n\}$. In other words,
\[
M_{z_i}^* M_{\Gamma} M_{z_j} M_{z_j}^* P_{\ker M_{z_i}^*}  M_{\Psi}^* M_{z_i}+ M_{z_i}^* M_{\Gamma} P_{\ker M_{z_j}^*}   P_{\ker M_{z_i}^*}  M_{\Psi}^* M_{z_i} = 0.
\]
If $j \neq i$, then we can use the commutator identity $[M_{z_j}, M_{z_i}^*] = 0$ to get
\[
M_{z_j} M_{z_i}^* M_{\Gamma}  P_{\ker M_{z_i}^*}  M_{\Psi}^* M_{z_i} M_{z_j}^* + M_{z_i}^* M_{\Gamma} P_{\ker M_{z_j}^*}   P_{\ker M_{z_i}^*}  M_{\Psi}^* M_{z_i} = 0.
\]
By our assumption $M_{z_i}^* M_{\Gamma}  P_{\ker M_{z_i}^*}  M_{\Psi}^* M_{z_i}=0$, and hence,
\[
M_{z_i}^* M_{\Gamma} P_{\ker M_{z_j}^*}   P_{\ker M_{z_i}^*}  M_{\Psi}^* M_{z_i} = 0.
\]
Continuing in the same manner as above, we find that the above condition implies that for any $k \in \{1,\ldots,n\} \setminus \{i,j\}$, we get
\begin{align*}
&M_{z_i}^* M_{\Gamma}M_{z_k} M_{z_k}^* P_{\ker M_{z_j}^*}   P_{\ker M_{z_i}^*}  M_{\Psi}^* M_{z_i}  \\
&+ M_{z_1}^* M_{\Gamma}  P_{\ker M_{z_k}^*} P_{\ker M_{z_j}^*}   P_{\ker M_{z_i}^*}  M_{\Psi}^* M_{z_i} \\
&= 0,
\end{align*}
which further implies that
\[
M_{z_i}^* M_{\Gamma} P_{\ker M_{z_k}^*} P_{\ker M_{z_j}^*}   P_{\ker M_{z_i}^*}  M_{\Psi}^* M_{z_i} = 0.
\]
Iterating the same procedure for $n-3$ many times with distinct numbers in $\{1,\ldots,n\} \setminus \{i,j,k\}$, we can conclude that,
\[
M_{z_i}^* M_{\Gamma} P_{\cle} M_{\Psi}^* M_{z_i} = 0,
\]
because $P_{\cle} = \underset{i=1}{\overset{n}{\Pi}} P_{\ker M_{z_i}^*}$. To prove the opposite direction, let us observe that,
\[
M_{z_i}^* M_{\Gamma} P_{\cle} M_{\Psi}^* M_{z_i} = 0,
\]
for some $i \in \{1,\ldots,n\}$ will imply that
\[
M_{z_i}^* M_{\Gamma} P_{\ker M_{z_1}^*} P_{\ker M_{z_i}^*} \big( \underset{k=2; k \neq i}{\overset{n}{\Pi}}  P_{\ker M_{z_k}^*} \big)  M_{\Psi}^* M_{z_i} = 0.
\]
This implies that
\begin{align*}
&M_{z_i}^* M_{\Gamma} P_{\ker M_{z_i}^*} \big( \underset{k=2; k \neq i}{\overset{n}{\Pi}} P_{\ker M_{z_k}^*} \big) M_{\Psi}^* M_{z_i} \\
&= M_{z_i}^* M_{\Gamma}  M_{z_1}M_{z_1}^*P_{\ker M_{z_i}^*} \big( \underset{k=2; k \neq i}{\overset{n}{\Pi}} P_{\ker M_{z_k}^*} \big)  M_{\Psi}^* M_{z_i}\\
&= M_{z_1} M_{z_i}^* M_{\Gamma} P_{\ker M_{z_i}^*} \big( \underset{k=2; k \neq i}{\overset{n}{\Pi}} P_{\ker M_{z_k}^*} \big) M_{\Psi}^* M_{z_i} M_{z_1}^*.
\end{align*}
and therefore, using recursion, for any $m \in \Nat$, we get
\begin{align*}
&M_{z_i}^* M_{\Gamma} P_{\ker M_{z_i}^*}\Pi_{k=2; k \neq i}^n P_{\ker M_{z_k}^*} M_{\Psi}^* M_{z_i} \\
&= M_{z_1}^m M_{z_i}^* M_{\Gamma} P_{\ker M_{z_i}^*}\Pi_{k=2; k \neq i}^n P_{\ker M_{z_k}^*} M_{\Psi}^* M_{z_i} M_{z_1}^{*m}.
\end{align*}
This will imply that for any $f \in H_{\cle}^2(\D^n)$,
\begin{align*}
&\| M_{z_i}^* M_{\Gamma} P_{\ker M_{z_i}^*}\Pi_{k=2; k \neq i}^n P_{\ker M_{z_k}^*} M_{\Psi}^* M_{z_i}  f\| \\
&= \|M_{z_1}^m M_{z_i}^* M_{\Gamma} P_{\ker M_{z_i}^*}\Pi_{k=2; k \neq i}^n P_{\ker M_{z_k}^*} M_{\Psi}^* M_{z_i} M_{z_1}^{*m}f\| \\
& \leq \|M_{z_1}^{*m} f\| \underset{m \raro \infty} \raro 0.
\end{align*}
Hence, we get
\begin{equation}\label{identity1}
M_{z_i}^* M_{\Gamma} P_{\ker M_{z_i}^*}\Pi_{k=2; k \neq i}^n P_{\ker M_{z_k}^*} M_{\Psi}^* M_{z_i} = 0.
\end{equation}
We will repeat the same process coordinate-wise.  For instance, the above identity (\ref{identity1}) implies that,
\[
M_{z_i}^* M_{\Gamma} P_{\ker M_{z_2}^*}P_{\ker M_{z_i}^*}\Pi_{k=3; k \neq i}^n P_{\ker M_{z_k}^*} M_{\Psi}^* M_{z_i}  = 0,
\]
which further gives, 
\begin{align*}
&M_{z_i}^* M_{\Gamma} P_{\ker M_{z_i}^*}\Pi_{k=3; k \neq i}^n P_{\ker M_{z_k}^*} M_{\Psi}^* M_{z_i}  \\
&= 
M_{z_2}M_{z_i}^* M_{\Gamma} P_{\ker M_{z_i}^*}\Pi_{k=3; k \neq i}^n P_{\ker M_{z_k}^*} M_{\Psi}^* M_{z_i}M_{z_2}^*.
\end{align*}
Again as above, for any $m \in \Nat$, we have
\begin{align*}
&M_{z_i}^* M_{\Gamma} P_{\ker M_{z_i}^*}\Pi_{k=3; k \neq i}^n P_{\ker M_{z_k}^*} M_{\Psi}^* M_{z_i} \\
&= 
M_{z_2}^m M_{z_i}^* M_{\Gamma} P_{\ker M_{z_i}^*}\Pi_{k=3; k \neq i}^n P_{\ker M_{z_k}^*} M_{\Psi}^* M_{z_i}M_{z_2}^{*m},
\end{align*}
which, as in the earlier case, will give
\[
M_{z_i}^* M_{\Gamma} P_{\ker M_{z_i}^*}\Pi_{k=3; k \neq i}^n P_{\ker M_{z_k}^*} M_{\Psi}^* M_{z_i}  = 0.
\]
Repeating the same process for all coordinates in $\{3,\ldots,n\} \setminus \{i\}$, will give us
\[
M_{z_i}^* M_{\Gamma} P_{\ker M_{z_i}^*} M_{\Psi}^* M_{z_i}  = 0.
\]
This completes the proof.
\end{proof}
We are now ready to establish conditions making $M_{\Gamma}M_{\Psi}^*$ into a Toeplitz operator.
\begin{propn}\label{main2}
Let $\clf, \cle$ be Hilbert spaces and $\Gamma(\bm{z}), \Psi(\bm{z})$ be $\clb(\clf,\cle)$-valued bounded analytic functions on $\D^n$. More precisely, let
\[
\Gamma(\bm{z}):= \sum_{\bm{l} \in \Nat^n} A_{\bm{l}} \bm{z}^{\bm{l}}  \in H_{\clb(\clf, \cle)}^{\infty}(\D^n); \quad \Psi(\bm{z}):= \sum_{\bm{m} \in \Nat^n} B_{\bm{m}} \bm{z}^{\bm{m}}  \in H_{\clb(\clf, \cle)}^{\infty}(\D^n),
\]
where $A_{\bm{k}}, B_{\bm{k}} \in \clb(\clf,\cle)$ for all $\bm{k} \in \Nat^n$, then $M_{\Gamma}M_{\Psi}^*$ is a Toeplitz operator on $H_{\cle}^2(\D^n)$ if and only if 
\begin{equation}\label{prod}
A_{\bm{l}+ e_i} B_{\bm{m}+ e_i}^* = 0 \quad (\forall \bm{l},\bm{m} \in \Nat^n),
\end{equation}
where $e_i := (0,\ldots,0, \underset{\text{i-th position}}{1},0,\ldots,0)$ for all $i \in \{1,\ldots,n\}$. 
\end{propn}
\begin{proof}
From Lemma \ref{toepcondn1} and \ref{toepcondn2}, we know that $M_{\Gamma}M_{\Psi}^*$ is a Toeplitz operator on $H_{\cle}^2(\D^n)$ if and only if
\[
T_{\overline{z_i} \Gamma} P_{\cle} T_{z_i \Psi^*} = M_{z_i}^* M_{\Gamma}P_{\cle} M_{\Psi}^* M_{z_i}=0,
\]
for all $i \in \{1,\ldots,n\}$. Now, for any $\bm{j},\bm{k} \in \Nat^n$ and $\eta, \zeta \in \cle$, we get
\[
\langle P_{\cle}  T_{z_i \Psi^*} \bm{z}^{\bm{k}} \eta, P_{\cle}  T_{z_i \Gamma^*} \bm{z}^{\bm{j}} \zeta \rangle = \langle P_{\cle}  T_{z_i^{k_i+1} \Psi^*} \bm{\hat{z}_i}^{\bm{\hat{k}_i}} \eta, P_{\cle} T_{z_i^{j_i+1} \Gamma^*} \bm{\hat{z}_i}^{\bm{\hat{j}_i}} \zeta \rangle,
\]
where for any $i \in \{1,\ldots,n\}$ and $\bm{m} \in \Nat^n$,
\[
\bm{\hat{z}_i}^{\bm{\hat{m}_i}}: =\underset{k=1; k \neq i}{ \overset{n}{\Pi}} z_k^{m_k}.
\]
By this notation, we have $\bm{z}^{\bm{m}} = z_i^{m_i} \bm{\hat{z}_i}^{\bm{\hat{m}_i}}$. Now, note that from our assumption
\[
\Gamma(\bm{z}):= \sum_{\bm{l} \in \Nat^n} \bm{z}^{\bm{l}} A_{\bm{l}}; \quad \Psi(\bm{z}):= \sum_{\bm{m} \in \Nat^n} \bm{z}^{\bm{m}} B_{\bm{m}},
\]
where $A_{\bm{l}}:= A_{(l_1, \ldots,l_n)}, B_{\bm{m}}:= B_{(m_1, \ldots,m_n)} \in \clb(\clf, \cle)$ for all $\bm{l} = (l_1, \ldots,l_n),$ and $\bm{m} = (m_1, \ldots,m_n)$ in $\Nat^n$.  Now,
\begin{align*}
P_{\cle} T_{{z_i}^{j_i+1} \Gamma^*} \bm{\hat{z}_i}^{\bm{\hat{j}_i}} \zeta 
&= P_{\cle} P_{H_{\cle}^2(\D^n)}{z_i}^{j_i+1}  \bm{\hat{z}_i}^{\bm{\hat{j}_i}}  \Gamma(\bm{z})^* \zeta \\
&= P_{\cle} \sum_{\bm{l} \in \Nat^n} {z_i}^{j_i+1} \bm{\hat{z}_i}^{\bm{\hat{j}_i}} \bar{\bm{z}}^{\bm{l}} A_{\bm{l}}^* \zeta \\
&= P_{\cle} \sum_{\bm{l} \in \Nat^n} {z_i}^{j_i - l_i +1} \bm{\hat{z}_i}^{\bm{\hat{j}_i}} \bar{\bm{\hat{z}_i}}^{\bm{\hat{l}_i}} A_{\bm{l}}^* \zeta \\
&= P_{\cle} \sum_{\bm{l} \in \Nat^n; l_i = j_i + 1} \bm{\hat{z}_i}^{\bm{\hat{j}_i - \hat{l}_i}}  A_{\bm{l}}^* \zeta \\
&=  A_{(j_1, \ldots, j_{i-1}, j_{i}+1, j_{i+1}, \ldots, j_n)}^* \zeta.
\end{align*}
Similarly, 
\[
P_{\cle}  T_{z_i \Psi^*} \bm{z}^{\bm{k}}  \eta = P_{\cle}  T_{z_i^{k_i+1} \Psi^*} \bm{\hat{z}_i}^{\bm{\hat{k}_i}} \eta = B_{(k_1,\ldots, k_{i-1}, k_i+1,k_{i+1}, \ldots, k_n)}^* \eta.
\]
Hence, 
\begin{align*}
&\langle P_{\cle}  T_{z_i \Psi^*} \bm{z}^{\bm{k}} \eta, P_{\cle}  T_{z_i \Gamma^*} \bm{z}^{\bm{j}} \zeta \rangle \\
&= \langle B_{(k_1,\ldots, k_{i-1}, k_i+1,k_{i+1}, \ldots, k_n)}^* \eta, A_{(j_1,\ldots, j_{i-1}, j_i+1,j_{i+1}, \ldots, j_n)}^* \zeta \rangle,
\end{align*}
for any $\eta, \zeta \in \cle$. Therefore, for any $i \in \{1,\ldots,n\}$, we get $T_{\overline{z_i} \Gamma} P_{\cle} T_{z_i \Psi^*}=0$ if and only if 
\[
A_{(j_1,\ldots, j_{i-1}, j_i+1,j_{i+1}, \ldots, j_n)} B_{(k_1,\ldots, k_{i-1}, k_{i+1},k_{i+1}, \ldots, k_n)}^* = 0,
\]
for all $\bm{j}, \bm{k} \in \Nat^n$. Using Lemma \ref{toepcondn1} and Lemma \ref{toepcondn2}, we get $M_{\Gamma}M_{\Psi}^*$ is a Toeplitz operator if and only if 
\[
A_{\bm{l}+e_{i}} B_{\bm{m}+e_i}^* = 0.
\]
for all $i \in \{1,\ldots,n\}$, and $\bm{l}, \bm{m} \in \Nat^n$.  
This completes the proof.
\end{proof}
We need the following result for the sequel.
\begin{lemma}\label{rule2}
Let $\clf, \cle$ be Hilbert spaces and $\Gamma(\bm{z}), \Psi(\bm{z})$ be $\clb(\clf,\cle)$-valued bounded analytic functions on $\D^n$.  Then $M_{\Gamma} M_{\Psi}^* = 0$ if and only if $\Gamma(\bm{\lambda}) \Psi(\bm{\mu})^* = 0$ for all $\bm{\lambda}, \bm{\mu} \in \D^n$.
\end{lemma}
\begin{proof}
It is well known that $\{k_{\bm{\mu}}(\cdot) \eta: \eta \in \cle, \bm{\mu} \in \D^n\}$ is a total set in $H_{\cle}^2(\D^n)$ and also, the Szeg\"o kernels $k_{\bm{\mu}}(\bm{\lambda})$ are never zero for any $\bm{\lambda}, \bm{\mu} \in \D^n$. Hence, the following identity
\[
\langle M_{\Gamma} M_{\Psi}^* k_{\bm{\mu}} \zeta, k_{\bm{\lambda}} \eta \rangle_{H_{\cle}^2(\D^n)} = \langle k_{\bm{\mu}} \Psi(\bm{\mu})^* \zeta, k_{\bm{\lambda}} \Gamma(\bm{\lambda})^* \eta \rangle_{H_{\cle}^2(\D^n)} = k_{\bm{\mu}}(\bm{\lambda}) \langle \Gamma(\bm{\lambda}) \Psi(\bm{\mu})^* \zeta,  \eta \rangle_{\cle},
\]
for any $\bm{\lambda}, \bm{\mu} \in \D^n$ and $\eta, \zeta \in \cle$, shows that $M_{\Gamma} M_{\Psi}^* = 0$ if and only if $ \Gamma(\bm{\lambda}) \Psi(\bm{\mu})^* = 0$ for all $\bm{\lambda}, \bm{\mu} \in \D^n$. This completes the proof.
\end{proof}
We are now ready to establish the main result of this section.
\begin{proof}[Proof of Theorem \ref{prod_mains}]
From Proposition \ref{main2}, we know that $M_{\Gamma} M_{\Psi}^*$ is a Toeplitz operator if and only if  $A_{\bm{l}+e_{i}} B_{\bm{m}+e_i}^* = 0$ for all $i \in \{1,\ldots,n\}$, and $\bm{l}, \bm{m} \in \Nat^n$, where 
\[
\Gamma(\bm{z}):= \sum_{\bm{l} \in \Nat^n} A_{\bm{l}} \bm{z}^{\bm{l}} ; \quad \Psi(\bm{z}):= \sum_{\bm{m} \in \Nat^n} B_{\bm{m}} \bm{z}^{\bm{m}} .
\]
Let us begin by assuming $M_{\Gamma} M_{\Psi}^*$ is a Toeplitz operator, equivalently, the above conditions hold for all $i \in \{1,\ldots,n\}$, and $\bm{l}, \bm{m} \in \Nat^n$,
then for any $k \in \{1,\ldots,n\}$, we get
\begin{equation*}
\begin{split}
&\big( \Gamma(\bm{\lambda}) - \Gamma_k(\bm{\lambda}) \big) \big( \Psi(\bm{\mu}) - \Psi_k(\bm{\mu})\big)^*\\
&=\Big(\Gamma\big((\lambda_1,\ldots,\lambda_n)\big) - \Gamma\big((\lambda_{1},\ldots,\lambda_{k-1}, 0, \lambda_{k+1},\ldots,\lambda_n)\big) \Big)\\
&\cdot \Big( \Psi\big((\mu_{1},\ldots,\mu_n)\big) - \Psi\big((\mu_{1},\ldots,\mu_{k-1}, 0, \mu_{k+1},\ldots,\mu_n) \big)\Big)^* \\
&= \sum_{\substack{\bm{l} \in \Nat^n; \\ l_{k} \neq 0}}  \bm{\lambda}^{\bm{l}} A_{\bm{l}}  \sum_{\substack{\bm{m} \in \Nat^n; \\ m_{k} \neq 0}} B_{\bm{m}}^* \bar{\bm{\mu}}^{\bm{m}} =0.
\end{split}
\end{equation*}
Now, we need to prove the converse direction. So, let us suppose
\[
\big( \Gamma(\bm{\lambda}) - \Gamma_k(\bm{\lambda}) \big) \big( \Psi(\bm{\mu}) - \Psi_k(\bm{\mu})\big)^*=0,
\]
for all $k \in \{1,\ldots,n\}$ and $\bm{\lambda}, \bm{\mu} \in \D^n$.  Using Lemma \ref{rule2}, we can conclude that the above conditions are equivalent to
\[
(M_{\Gamma} - M_{\Gamma_k}) (M_{\Psi} - M_{\Psi_k})^* = 0.
\]
Hence, for all $k \in \{1,\ldots,n\}$, we have
\[
M_{\Gamma}M_{\Psi}^* = M_{\Gamma_k} M_{\Psi}^* + M_{\Gamma} M_{\Psi_k}^* - M_{\Gamma_k}M_{\Psi_k}^*.
\]
Therefore, for all $k \in \{1,\ldots,n\}$,
\begin{align*}
M_{z_k}^* M_{\Gamma}M_{\Psi}^* M_{z_k} &= M_{z_k}^* M_{\Gamma_k} M_{\Psi}^* M_{z_k} + M_{z_k}^*M_{\Gamma} M_{\Psi_k}^* M_{z_k} - M_{z_k}^*M_{\Gamma_k}M_{\Psi_k}^* M_{z_k}\\
&= M_{\Gamma_k} M_{z_k}^* M_{\Psi}^* M_{z_k} + M_{z_k}^*M_{\Gamma} M_{z_k} M_{\Psi_k}^*  - M_{\Gamma_k} M_{z_k}^* M_{z_k} M_{\Psi_k}^* \\
&= M_{\Gamma_k} M_{\Psi}^* M_{z_k}^* M_{z_k} + M_{z_k}^* M_{z_k} M_{\Gamma} M_{\Psi_k}^* - M_{\Gamma_k} M_{z_k}^* M_{z_k} M_{\Psi_k}^* \\
&= M_{\Gamma_k} M_{\Psi}^* + M_{\Gamma} M_{\Psi_k}^* - M_{\Gamma_k} M_{\Psi_k}^*\\
&= M_{\Gamma} M_{\Psi}^*.
\end{align*}
The second and third equality follow from the fact that for any $k \in \{1,\ldots,n\}$, the functions $\Gamma_k, \Psi_k$ do not depend on the $z_k$-variable. Using Theorem \ref{Toep_char}, it is clear that the above set of identities will imply that $M_{\Gamma} M_{\Psi}^*$ must be a Toeplitz operator. This completes the proof.
\end{proof}
It is natural to ask what will be the symbol of the Toeplitz operator in the above case.  Let us start with the decomposition of $M_{\Gamma}M_{\Psi}^*$ in the case of $n=1$.
\begin{proof}[Proof of Theorem \ref{simple}]
From Theorem \ref{prod_mains} it is clear that $(i) \iff (ii)$.  Now suppose that condition $(ii)$ is true. By Lemma \ref{rule2}, condition $(ii)$ is equivalent to the following identity. 
\[
\big( M_{\Gamma} - M_{\Gamma(0)}\big) \big( M_{\Psi} - M_{\Psi(0)}\big)^* = 0.
\]
The above identity is again equivalent to the following decomposition.
\begin{align*}
M_{\Gamma} M_{\Psi}^* &=  M_{\Gamma}M_{\Psi(0)}^* + M_{\Gamma(0)}M_{\Psi}^* + M_{\Gamma(0)}M_{\Psi(0)}^*\\
&= M_{\Gamma} \Psi(0)^* + \Gamma(0) M_{\Psi}^* + \Gamma(0) \Psi(0)^*,
\end{align*}
It follows that $(ii) \iff (iii)$.  Also, the above decomposition directly implies that $M_{\Gamma} M_{\Psi}^*$ is a Toeplitz operator since the first, second and third terms on the right-hand side are analytic, co-analytic and constant Toeplitz operators, respectively.  This completes the proof.
\end{proof}
\begin{rem}\label{interes}
The above results unearths a surprising consequence: \textsf{if $\underset{l \in \Nat}{\sum} z^l A_{l}$ and $\underset{m \in \Nat}{\sum} z^m B_{m}$ are operator-valued bounded analytic functions on $\D$, then 
$
\big(\underset{\substack{l \in \Nat; \\ l \neq 0}}{\sum} {\lambda}^l A_{l} \big) \big(\underset{\substack{m \in \Nat; \\ m \neq 0}}{\sum} \bar{\mu}^m B_{m}^* \big) = 0,
$
for all $\lambda, \mu \in \D$ if and only if $A_{l+1} B_{m+1}^* = 0$, for all $l,m \in \Nat$.
}
In particular, the product of two series is zero if and only if certain term-wise coefficients are zero.
\end{rem}
For stating the theorem in cases when $n>1$, we need to set the following notation: \textsf{given any $k \in \{1,\ldots,n\}$, let $\{i_1,\ldots,i_k\} \subseteq \{1,\ldots,n\}$ be a non-empty subset of distinct elements in the ascending order. Given any $\Theta \in H_{\clb(\cle)}^{\infty}(\D^n)$, we denote the function which has $0$'s in the $\{i_1,\ldots,i_k\}$-coordinates by $\Theta_{i_1, \ldots, i_k}(\bm{z})$. For instance, $\Theta_{1, \ldots, n}(\bm{z}) = \Theta(0, \ldots,0)$ and $\Theta_{1,3,n}(\bm{z}) = \Theta(0,z_2,0,z_4,\ldots,z_{n-1}, 0)$, and $\Theta_{5,7,9}(\bm{z}) = \Theta(z_1,\ldots,z_4,0,z_6,0,z_8,0,\z_{10},\ldots,z_{n})$}. We also need the following result in the sequel.
\begin{lemma}\label{rule}
Let $\clf, \cle$ be Hilbert spaces and $\Gamma(\bm{z}), \Psi(\bm{z})$ be $\clb(\clf,\cle)$-valued bounded analytic functions on $\D^n$.  Furthermore, let $A,B \subseteq \{1,\ldots,n\}$ such that $A \cup B = \{1,\ldots,n\}$. Then $M_{\Gamma_A} M_{\Psi_B}^*$ is always a Toeplitz operator on $H_{\cle}^2(\D^n)$.
\end{lemma}
\begin{proof}
Suppose $i \in A$, then $\Gamma_A(\bm{z})$ is not dependent on the variable $i$. Hence,
\[
M_{z_i}^* M_{\Gamma_A} M_{\Psi_B}^* M_{z_i} =  M_{\Gamma_A} M_{z_i}^* M_{\Psi_B}^* M_{z_i} = M_{\Gamma_A} M_{\Psi_B}^* M_{z_i}^*  M_{z_i} = M_{\Gamma_A} M_{\Psi_B}^*.
\]
Now let us suppose $i \in B$. For a similar reasoning as above,
\[
M_{z_i}^* M_{\Gamma_A} M_{\Psi_B}^* M_{z_i} =  M_{z_i}^* M_{\Gamma_A} M_{z_i} M_{\Psi_B}^* = M_{z_i}^* M_{z_i} M_{\Gamma_A} M_{\Psi_B}^* = M_{\Gamma_A} M_{\Psi_B}^*.
\]
From our assumption $A \cup B = \{1,\ldots,n\}$,  it implies that $M_{z_i}^* M_{\Gamma_A} M_{\Psi_B}^* M_{z_i}= M_{\Gamma_A} M_{\Psi_B}^*$ for all $i \in \{1,\ldots,n\}$. Thus, by using Theorem \ref{Toep_char}  we can conclude that $M_{\Gamma_A}M_{\Psi_B}^*$ is a Toeplitz operator. This completes the proof.
\end{proof}
We are now ready to state the main theorem in the case of the bidisc.
\begin{thm}\label{bidisc}
Let $\clf, \cle$ be Hilbert spaces and $\Gamma(\bm{z}), \Psi(\bm{z})$ be $\clb(\clf,\cle)$-valued bounded analytic functions on $\D^2$. Then the following are equivalent \\ \vspace{1mm}
$(i)$~$M_{\Gamma}M_{\Psi}^*$ is a Toeplitz operator on $H_{\cle}^2(\D^2)$, \\ \vspace{1mm}
$(ii)$~$\big( \Gamma(\bm{\lambda}) - \Gamma_k(\bm{\lambda})\big) \big(  \Psi(\bm{\mu}) - \Psi_k(\bm{\mu}) \big)^* = 0$ for all $\bm{\lambda}, \bm{\mu} \in \D^2$ and $k=1,2$,\\ \vspace{1mm}
$(iii)$~$M_{\Gamma}M_{\Psi}^*$ admits the following decomposition,
\begin{align*}
M_{\Gamma}M_{\Psi}^*  &=  M_{\Gamma} \Psi(\bm{0})^*\\
&+M_{\Gamma_1} M_{\Psi_2}^* - M_{\Gamma_1} \Psi(\bm{0})^*\\
&+M_{\Gamma_2} M_{\Psi_1}^*  - M_{\Gamma_2} \Psi(\bm{0})^* \\
& + \Gamma(\bm{0}) M_{\Psi}^* + \Gamma(\bm{0}) M_{\Psi_1}^* + \Gamma(\bm{0}) M_{\Psi_2}^*- \Gamma(\bm{0}) \Psi(\bm{0})^*.
\end{align*}
\end{thm}
\begin{proof}
From Theorem \ref{prod_mains}, it is clear that $(i) \iff (ii)$.  Now suppose that condition $(ii)$ is true.  We will show that $M_{\Gamma}M_{\Psi}^*$ can be decomposed into a sum of elementary Toeplitz operators.  For this purpose, let us start with condition $(ii)$ when $k=1$, that is, 
\[
\big( \Gamma(\bm{\lambda}) - \Gamma_1(\bm{\lambda})\big) \big(  \Psi(\bm{\mu}) - \Psi_1(\bm{\mu}) \big)^* = 0.
\]
It implies that 
\begin{equation}\label{ca1}
 \Gamma(\bm{\lambda})\Psi(\bm{\mu})^* = \Gamma(\bm{\lambda})\Psi_1(\bm{\mu})^* + \Gamma_1({\bm{\lambda}})\Psi(\bm{\mu})^* - \Gamma_1(\bm{\lambda})\Psi_1(\bm{\mu})^*
\end{equation}
Again using condition $(ii)$, when $k=2$ and $\lambda_1 = 0$, we get
\[
\big( \Gamma_1(\bm{\lambda}) - \Gamma(\bm{0})\big) \big(  \Psi(\bm{\mu}) - \Psi_2(\bm{\mu}) \big)^* = 0.
\]
Hence, we get the following decomposition.
\[
\Gamma_1(\bm{\lambda})\Psi(\bm{\mu})^* =  \Gamma(\bm{0})\Psi(\bm{\mu})^* + \Gamma_{1}(\bm{\lambda})\Psi_2(\bm{\mu})^*- \Gamma(\bm{0})\Psi_2(\bm{\mu})^*.
\]
Since the above condition is assumed to be true for all $\bm{\lambda}, \bm{\mu}$ in $\D^2$, we can put $\mu_1 = 0$ to get
\[
 \Gamma_1(\bm{\lambda})\Psi_1(\bm{\mu})^* =  \Gamma(\bm{0})\Psi_1(\bm{\mu})^* +  \Gamma_1(\bm{\lambda}) \Psi(\bm{0})^* - \Gamma(\bm{0}) \Psi(\bm{0})^* .
\]
Incorporating the above decompositions in condition (\ref{ca1}), we get
\begin{equation}\label{ca3}
\begin{split}
 \Gamma(\bm{\lambda})\Psi(\bm{\mu})^* &= \Gamma(\bm{\lambda})\Psi_1(\bm{\mu})^* 
+\Gamma(\bm{0})\Psi(\bm{\mu})^* + \Gamma_{1}(\bm{\lambda})\Psi_2(\bm{\mu})^*- \Gamma(\bm{0})\Psi_2(\bm{\mu})^*\\
&-\Gamma(\bm{0})\Psi_1(\bm{\mu})^* - \Gamma_1(\bm{\lambda}) \Psi(\bm{0})^* + \Gamma(\bm{0}) \Psi(\bm{0})^* .
\end{split}
\end{equation}
Using condition $(ii)$,when $k=2$ and $\mu_1 = 0$, we get
\[
\big(\Gamma(\bm{\lambda}) - \Gamma_{2}(\bm{\lambda})\big) \big( \Psi_1({\bm{\mu}}) - \Psi(\bm{0}) \big)^* = 0.
\]
This gives the following decomposition.
\[
\Gamma(\bm{\lambda})\Psi_1({\bm{\mu}})^* = \Gamma(\bm{\lambda}) \Psi(\bm{0})^* 
+\Gamma_{2}(\bm{\lambda}) \Psi_1({\bm{\mu}})^* 
- \Gamma_{2}(\bm{\lambda})\Psi(\bm{0})^*.
\]
Using the above decomposition in equation (\ref{ca3}), we get
\begin{equation*}
\begin{split}
 \Gamma(\bm{\lambda})\Psi(\bm{\mu})^* &= \Gamma(\bm{\lambda}) \Psi(\bm{0})^* 
+\Gamma_{2}(\bm{\lambda}) \Psi_1({\bm{\mu}})^* 
- \Gamma_{2}(\bm{\lambda})\Psi(\bm{0})^*
+\Gamma (\bm{0}) \Psi(\bm{\mu})^* 
+ \Gamma_{1}(\bm{\lambda})\Psi_2(\bm{\mu})^*\\
&- \Gamma(\bm{0})\Psi_2(\bm{\mu})^*
- \Gamma(\bm{0})\Psi_1(\bm{\mu})^* 
-\Gamma_1(\bm{\lambda}) \Psi(\bm{0})^* 
+\Gamma(\bm{0})\Psi(\bm{0})^* .
\end{split}
\end{equation*}
Since the above is true for all $\bm{\lambda}, \bm{\mu} \in \D^2$, we get for any $\eta, \zeta \in \cle$, 
\begin{align*}
\langle M_{\Gamma}M_{\Psi}^* k_{\bm{\mu}} \eta, k_{\bm{\lambda}} \zeta \rangle &= k_{\bm{\mu}}(\bm{\lambda}) \langle  \Gamma(\bm{\lambda})\Psi(\bm{\mu})^*  \eta,  \zeta \rangle \\
&= k_{\bm{\mu}}(\bm{\lambda}) \Big( \langle \Gamma(\bm{\lambda})  \Psi(\bm{0})^* \eta,  \zeta \rangle
+ \langle \Gamma_{2}(\bm{\lambda}) \Psi_1({\bm{\mu}})^* \eta,  \zeta \rangle
- \langle \Gamma_{2}(\bm{\lambda})\Psi(\bm{0})^* \eta,  \zeta \rangle\\
&+ \langle \Gamma(\bm{0})\Psi(\bm{\mu})^* \eta,  \zeta \rangle 
+ \langle \Gamma_{1}(\bm{\lambda})\Psi_2(\bm{\mu})^* \eta,  \zeta \rangle - \langle \Gamma(\bm{0})\Psi_2(\bm{\mu})^* \eta,  \zeta \rangle \\
&- \langle \Gamma(\bm{0})\Psi_1(\bm{\mu})^* \eta,  \zeta \rangle 
- \langle \Gamma_1(\bm{\lambda}) \Psi(\bm{0})^* \eta,  \zeta \rangle
+ \langle \Gamma(\bm{0})\Psi(\bm{0})^* \eta,  \zeta \rangle \Big) \\
&=  \langle M_{\Gamma} \Psi(\bm{0})^* k_{\bm{\mu}} \eta, k_{\bm{\lambda}} \zeta \rangle\\
&+\langle M_{\Gamma_1} M_{\Psi_2}^* k_{\bm{\mu}} \eta, k_{\bm{\lambda}} \zeta \rangle - \langle M_{\Gamma_1} \Psi(\bm{0})^* k_{\bm{\mu}} \eta, k_{\bm{\lambda}} \zeta \rangle\\
&+ \langle M_{\Gamma_2} M_{\Psi_1}^*k_{\bm{\mu}} \eta, k_{\bm{\lambda}} \zeta \rangle  - M_{\Gamma_2} \Psi(\bm{0})^*k_{\bm{\mu}} \eta, k_{\bm{\lambda}} \zeta \rangle \\
& + \langle \Gamma(\bm{0}) M_{\Psi}^*k_{\bm{\mu}} \eta, k_{\bm{\lambda}} \zeta \rangle + \langle \Gamma(\bm{0}) M_{\Psi_1}^*k_{\bm{\mu}} \eta, k_{\bm{\lambda}} \zeta \rangle \\
&+ \langle \Gamma(\bm{0}) M_{\Psi_2}^*k_{\bm{\mu}} \eta, k_{\bm{\lambda}} \zeta \rangle - \langle \Gamma(\bm{0}) \Psi(\bm{0})^*k_{\bm{\mu}} \eta, k_{\bm{\lambda}} \zeta \rangle,
\end{align*}
and therefore,
\begin{align*}
M_{\Gamma}M_{\Psi}^*  &=  M_{\Gamma} \Psi(\bm{0})^*\\
&+M_{\Gamma_1} M_{\Psi_2}^* - M_{\Gamma_1} \Psi(\bm{0})^*\\
&+M_{\Gamma_2} M_{\Psi_1}^*  - M_{\Gamma_2} \Psi(\bm{0})^* \\
& + \Gamma(\bm{0}) M_{\Psi}^* + \Gamma(\bm{0}) M_{\Psi_1}^* + \Gamma(\bm{0}) M_{\Psi_2}^*- \Gamma(\bm{0}) \Psi(\bm{0})^*.
\end{align*}
This completes the proof for the direction $(ii) \implies (iii)$. For the direction $(iii) \implies (i)$, it is easy to observe using Lemma \ref{rule} that all the operators on the right-hand side of the above decomposition are Toeplitz operators. Thus, $M_{\Gamma}M_{\Psi}^*$ must be a Toeplitz operator as well. This completes the proof.
\end{proof}

In the case of $n=3$, we get the following characterization.

\begin{thm}\label{tridisc}
Let $\clf, \cle$ be Hilbert spaces and $\Gamma(\bm{z}), \Psi(\bm{z})$ be $\clb(\clf,\cle)$-valued bounded analytic functions on $\D^3$. Then the following are equivalent \\ \vspace{1mm}
$(i)$~$M_{\Gamma}M_{\Psi}^*$ is a Toeplitz operator on $H_{\cle}^2(\D^3)$, \\ \vspace{1mm}
$(ii)$~$\big( \Gamma(\bm{\lambda}) - \Gamma_k(\bm{\lambda})\big) \big(  \Psi(\bm{\mu}) - \Psi_k(\bm{\mu}) \big)^* = 0$ for all $\bm{\lambda}, \bm{\mu} \in \D^3$ and $k=1,2,3$,\\ \vspace{1mm}
$(iii)$~$M_{\Gamma}M_{\Psi}^*$ admits the following decomposition, \vspace{1mm}
\begin{align*}
M_{\Gamma}M_{\Psi}^*  &=  M_{\Gamma} \Psi(\bm{0})^*\\
&-M_{\Gamma_1} \Psi(\bm{0})^* + M_{\Gamma_1} M_{\Psi_{2,3}}^*\\
&-M_{\Gamma_2} \Psi(\bm{0})^* + M_{\Gamma_2} M_{\Psi_{1,3}}^*\\
&-M_{\Gamma_3} \Psi(\bm{0})^* + M_{\Gamma_3} M_{\Psi_{1,2}}^*\\
&+M_{\Gamma_{1,2}} \Psi(\bm{0})^*  - M_{\Gamma_{1,2}} M_{\Psi_{1,3}}^* -  M_{\Gamma_{1,2}} M_{\Psi_{2,3}}^* + M_{\Gamma_{1,2}} M_{\Psi_{3}}^*\\
&+M_{\Gamma_{1,3}} \Psi(\bm{0})^*  - M_{\Gamma_{1,3}} M_{\Psi_{1,2}}^* -  M_{\Gamma_{1,3}} M_{\Psi_{2,3}}^* + M_{\Gamma_{1,3}} M_{\Psi_{2}}^*\\
&+M_{\Gamma_{2,3}} \Psi(\bm{0})^*  - M_{\Gamma_{2,3}} M_{\Psi_{1,2}}^* - M_{\Gamma_{2,3}} M_{\Psi_{1,3}}^* + M_{\Gamma_{2,3}} M_{\Psi_1}^* \\
&+\Gamma(\bm{0}) M_{\Psi}^*  - \Gamma(\bm{0}) M_{\Psi_{1}}^* -  \Gamma(\bm{0}) M_{\Psi_{2}}^* - \Gamma(\bm{0}) M_{\Psi_{3}}^*\\
&+\Gamma(\bm{0}) M_{\Psi_{1,2}}^*  - \Gamma(\bm{0}) M_{\Psi_{1,3}}^* -  \Gamma(\bm{0}) M_{\Psi_{2,3}}^* - \Gamma(\bm{0}) \Psi(\bm{0})^*.
\end{align*}
\end{thm}
\begin{proof}
The $(i) \iff (ii)$ direction is the same as in the previous theorem.  For the direction $(ii) \implies (iii)$, let us start with condition $(ii)$, when $k=1$, that is,
\[
\big( \Gamma(\bm{\lambda}) - \Gamma_1(\bm{\lambda}) \big) \big( \Psi(\bm{\mu}) - \Psi_1(\bm{\mu}) \big)^* = 0.
\]
This gives the following decomposition.
\begin{equation}\label{da1}
\Gamma(\bm{\lambda})  \Psi(\bm{\mu})^* =  \Gamma(\bm{\lambda}) \Psi_1(\bm{\mu})^* + \Gamma_1(\bm{\lambda})\Psi(\bm{\mu})^* - \Gamma_1(\bm{\lambda}) \Psi_1(\bm{\mu})^*.
\end{equation}
Again using condition $(ii)$ when $k=2$ and $\mu_1 = 0$, we get
\[
\big( \Gamma(\bm{\lambda}) - \Gamma_2(\bm{\lambda}) \big) \big( \Psi_{1}(\bm{\mu}) - \Psi_{1,2}(\bm{\mu}) \big)^* = 0.
\]
This implies that,
\[
 \Gamma(\bm{\lambda}) \Psi_1(\bm{\mu})^* =   \Gamma(\bm{\lambda}) \Psi_{1,2}(\bm{\mu})^* + \Gamma_2(\bm{\lambda}) \Psi_1(\bm{\mu})^* - \Gamma_2(\bm{\lambda}) \Psi_{1,2}(\bm{\mu})^* 
\]
Since the above condition is assumed to be true for all $\bm{\lambda}, \bm{\mu}$ in $\D^3$, we can put $\lambda_1 = 0$ to get
\[
\Gamma_1(\bm{\lambda}) \Psi_{1}(\bm{\mu})^* = \Gamma_1(\bm{\lambda}) \Psi_{1,2}(\bm{\mu})^*
+ \Gamma_{1,2}(\bm{\lambda})  \Psi_{1}(\bm{\mu})^* - \Gamma_{1,2}(\bm{\lambda}) \Psi_{1,2}(\bm{\mu})^*.
\]
Using condition $(ii)$, when $k=2$ and $\lambda_1 = 0$ gives
\[
\big( \Gamma_1(\bm{\lambda}) - \Gamma_{1,2}(\bm{\lambda}) \big) \big( \Psi(\bm{\mu}) - \Psi_{2}(\bm{\mu}) \big)^* = 0.
\]
\begin{equation*}
\Gamma_1(\bm{\lambda})\Psi(\bm{\mu})^* =  \Gamma_1(\bm{\lambda})\Psi_2(\bm{\mu})^* + \Gamma_{1,2}(\bm{\lambda}) \Psi(\bm{\mu})^* - \Gamma_{1,2}(\bm{\lambda}) \Psi_{2}(\bm{\mu})^*.
\end{equation*}
Incorporating all the above three decompositions inside condition (\ref{da1}), we get
\begin{equation}\label{da2}
\begin{split}
\Gamma(\bm{\lambda})  \Psi(\bm{\mu})^* &=  \Gamma(\bm{\lambda}) \Psi_{1,2}(\bm{\mu})^* + \Gamma_2(\bm{\lambda}) \Psi_1(\bm{\mu})^* - \Gamma_2(\bm{\lambda}) \Psi_{1,2}(\bm{\mu})^* \\
&+ \Gamma_1(\bm{\lambda})\Psi_2(\bm{\mu})^* + \Gamma_{1,2}(\bm{\lambda}) \Psi(\bm{\mu})^* - \Gamma_{1,2}(\bm{\lambda}) \Psi_{2}(\bm{\mu})^* \\
&- \Gamma_1(\bm{\lambda}) \Psi_{1,2}(\bm{\mu})^*
- \Gamma_{1,2}(\bm{\lambda})  \Psi_{1}(\bm{\mu})^* + \Gamma_{1,2}(\bm{\lambda}) \Psi_{1,2}(\bm{\mu})^*.
\end{split}
\end{equation}
Using condition $(ii)$, when $k=3$ and $\mu_1=\mu_2=0$, we get
\[
\big( \Gamma(\bm{\lambda}) - \Gamma_3(\bm{\lambda}) \big) \big( \Psi_{1,2}(\bm{\mu}) - \Psi (\bm{0}) \big)^* = 0.
\]
This gives the following decomposition.
\[
 \Gamma(\bm{\lambda})\Psi_{1,2}(\bm{\mu})^* =    \Gamma_3(\bm{\lambda}) \Psi_{1,2}(\bm{\mu})^* + \Gamma(\bm{\lambda}) \Psi(\bm{0})^* - \Gamma_3(\bm{\lambda}) \Psi(\bm{0})^*.
\]
Again using condition $(ii)$, when $k=3$ and $\lambda_1 = \lambda_2 = 0$, we get
\[
\big( \Gamma_{1,2}(\bm{\lambda}) - \Gamma (\bm{0}) \big) \big( \Psi(\bm{\mu}) - \Psi_3(\bm{\mu}) \big)^* = 0.
\]
This gives the following decomposition.
\[
\Gamma_{1,2}(\bm{\lambda}) \Psi(\bm{\mu})^* = \Gamma_{1,2}(\bm{\lambda})\Psi_{3}(\bm{\mu})^* + \Gamma(\bm{0})\Psi(\bm{\mu})^* - \Gamma(\bm{0})  \Psi_{3}(\bm{\mu})^*.
\]
Incorporating the above decompositions in condition (\ref{da2}), we get
\begin{equation}\label{da3}
\begin{split}
\Gamma(\bm{\lambda})  \Psi(\bm{\mu})^* &=  \Gamma_3(\bm{\lambda}) \Psi_{1,2}(\bm{\mu})^* + \Gamma(\bm{\lambda}) \Psi(\bm{0})^* - \Gamma_3(\bm{\lambda}) \Psi(\bm{0})^* \\
&+ \Gamma_2(\bm{\lambda}) \Psi_1(\bm{\mu})^* - \Gamma_2(\bm{\lambda}) \Psi_{1,2}(\bm{\mu})^* \\
&+ \Gamma_1(\bm{\lambda})\Psi_2(\bm{\mu})^*\\ 
&+ \Gamma_{1,2}(\bm{\lambda})\Psi_{3}(\bm{\mu})^* + \Gamma(\bm{0})\Psi(\bm{\mu})^* - \Gamma(\bm{0})  \Psi_{3}(\bm{\mu})^*\\
&- \Gamma_{1,2}(\bm{\lambda}) \Psi_{2}(\bm{\mu})^* \\
&- \Gamma_1(\bm{\lambda}) \Psi_{1,2}(\bm{\mu})^*
- \Gamma_{1,2}(\bm{\lambda})  \Psi_{1}(\bm{\mu})^* + \Gamma_{1,2}(\bm{\lambda}) \Psi_{1,2}(\bm{\mu})^*.
\end{split}
\end{equation}
Condition $(ii)$ when $k=3$, and $\lambda_1 = \mu_2 = 0$, gives us
\[
\big( \Gamma_{1}(\bm{\lambda}) - \Gamma_{1,3}(\bm{\lambda}) \big) \big( \Psi_2(\bm{\mu}) - \Psi_{2,3}(\bm{\mu}) \big)^* = 0,
\]
which provides the following decomposition.
\[
\Gamma_1(\bm{\lambda}) \Psi_{2}(\bm{\mu})^* = \Gamma_1(\bm{\lambda})  \Psi_{2,3}(\bm{\mu})^* +  \Gamma_{1,3}(\bm{\lambda}) \Psi_{2}(\bm{\mu})^* - \Gamma_{1,3}(\bm{\lambda}) \Psi_{2,3}(\bm{\mu})^*.
\]
Now, we will separately consider the following cases: $(a)~\mu_1 = 0 $, $(b)~\lambda_2 = 0$ and $(c)~\lambda_2 = \mu_1 = 0$, in the above condition to get the following three decompositions.
\[
\Gamma_1(\bm{\lambda})\Psi_{1,2}(\bm{\mu})^* =  \Gamma_1(\bm{\lambda}) \Psi(\bm{0})^* +  \Gamma_{1,3}(\bm{\lambda}) \Psi_{1,2}(\bm{\mu})^* - \Gamma_{1,3}(\bm{\lambda}) \Psi(\bm{0})^*. \tag{\text{a}}
\]
\[
\Gamma_{1,2}(\bm{\lambda}) \Psi_{2}(\bm{\mu})^* =   \Gamma_{1,2}(\bm{\lambda}) \Psi_{2,3}(\bm{\mu})^* + \Gamma(\bm{0}) \Psi_{2}(\bm{\mu})^* - \Gamma(\bm{0}) \Psi_{2,3}(\bm{\mu})^*. \tag{\text{b}}
\]
\[
\Gamma_{1,2}(\bm{\lambda})  \Psi_{1,2}(\bm{\mu})^* = \Gamma_{1,2}(\bm{\lambda}) \Psi(\bm{0})^* + \Gamma(\bm{0}) \Psi_{1,2}(\bm{\mu})^* - \Gamma(\bm{0}) \Psi(\bm{0})^*. \tag{\text{c}}
\]
Incorporating the above identities in condition (\ref{da3}), we get
\begin{equation}\label{da4}
\begin{split}
\Gamma(\bm{\lambda})  \Psi(\bm{\mu})^* &=  \Gamma_3(\bm{\lambda}) \Psi_{1,2}(\bm{\mu})^* + \Gamma(\bm{\lambda}) \Psi(\bm{0})^* - \Gamma_3(\bm{\lambda}) \Psi(\bm{0})^* \\
&+ \Gamma_2(\bm{\lambda}) \Psi_1(\bm{\mu})^* - \Gamma_2(\bm{\lambda}) \Psi_{1,2}(\bm{\mu})^* \\
&+ \Gamma_1(\bm{\lambda})  \Psi_{2,3}(\bm{\mu})^* +  \Gamma_{1,3}(\bm{\lambda}) \Psi_{2}(\bm{\mu})^* - \Gamma_{1,3}(\bm{\lambda}) \Psi_{2,3}(\bm{\mu})^* \\
&+ \Gamma_{1,2}(\bm{\lambda})\Psi_{3}(\bm{\mu})^* + \Gamma(\bm{0})\Psi(\bm{\mu})^* - \Gamma(\bm{0})  \Psi_{3}(\bm{\mu})^*\\
&- \Gamma_{1,2}(\bm{\lambda}) \Psi_{2,3}(\bm{\mu})^* - \Gamma(\bm{0}) \Psi_{2}(\bm{\mu})^* + \Gamma(\bm{0}) \Psi_{2,3}(\bm{\mu})^* \\
&- \Gamma_1(\bm{\lambda}) \Psi(\bm{0})^* - \Gamma_{1,3}(\bm{\lambda}) \Psi_{1,2}(\bm{\mu})^* + \Gamma_{1,3}(\bm{\lambda}) \Psi(\bm{0})^*.\\
&- \Gamma_{1,2}(\bm{\lambda})  \Psi_{1}(\bm{\mu})^* \\
&+ \Gamma_{1,2}(\bm{\lambda}) \Psi(\bm{0})^* + \Gamma(\bm{0}) \Psi_{1,2}(\bm{\mu})^* - \Gamma(\bm{0}) \Psi(\bm{0})^*.
\end{split}
\end{equation}
In a similar manner, using condition $(ii)$, when $k=3$ and $\lambda_2 = \mu_1 = 0$ gives us 
\[
\big( \Gamma_{2}(\bm{\lambda}) - \Gamma_{2,3}(\bm{\lambda}) \big) \big( \Psi_1(\bm{\mu}) - \Psi_{1,3}(\bm{\mu}) \big)^* = 0,
\]
and hence,  the following decomposition.
\[
\Gamma_2(\bm{\lambda}) \Psi_{1}(\bm{\mu})^* =  \Gamma_2(\bm{\lambda})\Psi_{1,3}(\bm{\mu})^* + \Gamma_{2,3}(\bm{\lambda}) \Psi_{1}(\bm{\mu})^* - \Gamma_{2,3}(\bm{\lambda}) \Psi_{1,3}(\bm{\mu})^*. 
\]
Considering the following cases separately: $(d)~\lambda_1 =0 $, $(e)~\mu_2 = 0$ in the above condition gives us the following two decompositions.
\[
\Gamma_{1,2}(\bm{\lambda}) \Psi_{1}(\bm{\mu})^* =  \Gamma_{1,2}(\bm{\lambda})\Psi_{1,3}(\bm{\mu})^* + \Gamma (\bm{0}) \Psi_{1}(\bm{\mu})^* - \Gamma (\bm{0}) \Psi_{1,3}(\bm{\mu})^*. \tag{\text{d}}
\]
\[
\Gamma_2(\bm{\lambda}) \Psi_{1,2}(\bm{\mu})^* =  \Gamma_2(\bm{\lambda}) \Psi(\bm{0})^*+ \Gamma_{2,3}(\bm{\lambda}) \Psi_{1,2}(\bm{\mu})^* - \Gamma_{2,3}(\bm{\lambda}) \Psi(\bm{0})^*.  \tag{\text{e}}
\]
Incorporating the above decompositions inside condition (\ref{da4}), gives us
\begin{equation}\label{da5}
\begin{split}
\Gamma(\bm{\lambda})  \Psi(\bm{\mu})^* &=  \Gamma_3(\bm{\lambda}) \Psi_{1,2}(\bm{\mu})^* + \Gamma(\bm{\lambda}) \Psi(\bm{0})^* - \Gamma_3(\bm{\lambda}) \Psi(\bm{0})^* \\
&+ \Gamma_2(\bm{\lambda})\Psi_{1,3}(\bm{\mu})^* + \Gamma_{2,3}(\bm{\lambda}) \Psi_{1}(\bm{\mu})^* - \Gamma_{2,3}(\bm{\lambda}) \Psi_{1,3}(\bm{\mu})^*. \\
&- \Gamma_2(\bm{\lambda}) \Psi(\bm{0})^*- \Gamma_{2,3}(\bm{\lambda}) \Psi_{1,2}(\bm{\mu})^* + \Gamma_{2,3}(\bm{\lambda}) \Psi(\bm{0})^*\\
&+ \Gamma_1(\bm{\lambda})  \Psi_{2,3}(\bm{\mu})^* +  \Gamma_{1,3}(\bm{\lambda}) \Psi_{2}(\bm{\mu})^* - \Gamma_{1,3}(\bm{\lambda}) \Psi_{2,3}(\bm{\mu})^* \\
&+ \Gamma_{1,2}(\bm{\lambda})\Psi_{3}(\bm{\mu})^* + \Gamma(\bm{0})\Psi(\bm{\mu})^* - \Gamma(\bm{0})  \Psi_{3}(\bm{\mu})^*\\
\
&- \Gamma_{1,2}(\bm{\lambda}) \Psi_{2,3}(\bm{\mu})^* - \Gamma(\bm{0}) \Psi_{2}(\bm{\mu})^* + \Gamma(\bm{0}) \Psi_{2,3}(\bm{\mu})^* \\
&- \Gamma_1(\bm{\lambda}) \Psi(\bm{0})^* - \Gamma_{1,3}(\bm{\lambda}) \Psi_{1,2}(\bm{\mu})^* + \Gamma_{1,3}(\bm{\lambda}) \Psi(\bm{0})^*.\\
&- \Gamma_{1,2}(\bm{\lambda})\Psi_{1,3}(\bm{\mu})^* - \Gamma (\bm{0}) \Psi_{1}(\bm{\mu})^* + \Gamma (\bm{0}) \Psi_{1,3}(\bm{\mu})^*. \\
&+ \Gamma_{1,2}(\bm{\lambda}) \Psi(\bm{0})^* + \Gamma(\bm{0}) \Psi_{1,2}(\bm{\mu})^* - \Gamma(\bm{0}) \Psi(\bm{0})^*.
\end{split}
\end{equation}
Therefore, for any $\eta, \zeta \in \cle$, we get
\begin{align*}
\langle M_{\Gamma}M_{\Psi}^* k_{\bm{\mu}} \eta, k_{\bm{\lambda}} \zeta \rangle &= k_{\bm{{\mu}}}(\bm{\lambda}) \langle  \Gamma(\bm{\lambda})\Psi(\bm{\mu})^*  \eta,  \zeta \rangle \\
&= k_{\bm{{\mu}}}(\bm{\lambda}) \Big( \langle  \Gamma_3(\bm{\lambda}) \Psi_{1,2}(\bm{\mu})^* \eta,  \zeta \rangle + \langle  \Gamma(\bm{\lambda}) \Psi(\bm{0})^* \eta,  \zeta \rangle - \langle  \Gamma_3(\bm{\lambda}) \Psi(\bm{0})^* \eta,  \zeta \rangle \\
&+ \langle  \Gamma_2(\bm{\lambda})\Psi_{1,3}(\bm{\mu})^* \eta,  \zeta \rangle + \langle  \Gamma_{2,3}(\bm{\lambda}) \Psi_{1}(\bm{\mu})^* \eta,  \zeta \rangle - \langle \Gamma_{2,3}(\bm{\lambda}) \Psi_{1,3}(\bm{\mu})^* \eta,  \zeta \rangle. \\
&- \langle \Gamma_2(\bm{\lambda}) \Psi(\bm{0})^* \eta,  \zeta \rangle - \langle \Gamma_{2,3}(\bm{\lambda}) \Psi_{1,2}(\bm{\mu})^* \eta,  \zeta \rangle + \langle \Gamma_{2,3}(\bm{\lambda}) \Psi(\bm{0})^* \eta,  \zeta \rangle. \\
&+ \langle \Gamma_1(\bm{\lambda}) \Psi_{2,3}(\bm{\mu})^* \eta,  \zeta \rangle +  \langle \Gamma_{1,3}(\bm{\lambda}) \Psi_{2}(\bm{\mu})^* \eta,  \zeta \rangle - \langle \Gamma_{1,3}(\bm{\lambda}) \Psi_{2,3}(\bm{\mu})^* \eta,  \zeta \rangle \\
&+ \langle \Gamma_{1,2}(\bm{\lambda})\Psi_{3}(\bm{\mu})^* \eta,  \zeta \rangle + \langle \Gamma(\bm{0})\Psi(\bm{\mu})^* \eta,  \zeta \rangle - \langle \Gamma(\bm{0})  \Psi_{3}(\bm{\mu})^* \eta,  \zeta \rangle\\
&- \langle \Gamma_{1,2}(\bm{\lambda}) \Psi_{2,3}(\bm{\mu})^* \eta,  \zeta \rangle - \langle \Gamma(\bm{0}) \Psi_{2}(\bm{\mu})^* \eta,  \zeta \rangle + \langle \Gamma(\bm{0}) \Psi_{2,3}(\bm{\mu})^* \eta,  \zeta \rangle \\
&- \langle \Gamma_1(\bm{\lambda}) \Psi(\bm{0})^* \eta,  \zeta \rangle - \langle \Gamma_{1,3}(\bm{\lambda}) \Psi_{1,2}(\bm{\mu})^* \eta,  \zeta \rangle + \langle \Gamma_{1,3}(\bm{\lambda}) \Psi(\bm{0})^* \eta,  \zeta \rangle.\\
&- \langle \Gamma_{1,2}(\bm{\lambda})\Psi_{1,3}(\bm{\mu})^* \eta,  \zeta \rangle - \langle \Gamma (\bm{0}) \Psi_{1}(\bm{\mu})^* \eta,  \zeta \rangle + \langle \Gamma (\bm{0}) \Psi_{1,3}(\bm{\mu})^* \eta,  \zeta \rangle . \\
&+ \langle \Gamma_{1,2}(\bm{\lambda}) \Psi(\bm{0})^* \eta,  \zeta \rangle + \langle \Gamma(\bm{0}) \Psi_{1,2}(\bm{\mu})^* \eta,  \zeta \rangle - \langle \Gamma(\bm{0}) \Psi(\bm{0})^* \eta,  \zeta \rangle \Big).
\end{align*}
Since the above is true for all $\bm{\lambda}, \bm{\mu} \in \D^3$, and $\eta, \zeta \in \cle$, we get
\begin{align*}
M_{\Gamma}M_{\Psi}^*  &=  M_{\Gamma_3}M_{\Psi_{1,2}}^* + M_{\Gamma} \Psi(\bm{0})^* - M_{\Gamma_3} \Psi(\bm{0})^*\\
&+ M_{\Gamma_2} M_{\Psi_{1,3}}^* + M_{\Gamma_{2,3}} M_{\Psi_{1}}^* - M_{\Gamma_{2,3}}  M_{\Psi_{2,3}}^*\\
&-M_{\Gamma_2} \Psi(\bm{0})^* - M_{\Gamma_{2,3}} M_{\Psi_{1,2}}^* + M_{\Gamma_{2,3}}  \Psi(\bm{0})^*\\
&+M_{\Gamma_{1,2}} M_{\Psi_3}^* + \Gamma(\bm{0}) M_{\Psi}^* - \Gamma(\bm{0}) M_{\Psi_3}^*\\
&-M_{\Gamma_{1,2}} M_{\Psi_{2,3}}^*  - \Gamma(\bm{0}) M_{\Psi_{2}}^* +  \Gamma(\bm{0}) M_{\Psi_{2,3}}^*\\
&-M_{\Gamma_{1}} \Psi(\bm{0})^*  - M_{\Gamma_{1,3}} M_{\Psi_{1,2}}^* +  M_{\Gamma_{1,3}} \Psi(\bm{0})^* \\
&-M_{\Gamma_{1,2}} M_{\Psi_{1,3}}^*  - \Gamma(\bm{0}) M_{\Psi_{1}}^* +  \Gamma(\bm{0}) M_{\Psi_{1,3}}^* \\
&+M_{\Gamma_{1,2}} \Psi(\bm{0})^*  + \Gamma(\bm{0}) M_{\Psi_{1,2}}^* -  \Gamma(\bm{0}) \Psi(\bm{0})^*.
\end{align*}
This completes the proof for the direction $(ii) \implies (iii)$. For the direction $(iii) \implies (i)$, it is easy to observe using Lemma \ref{rule} that all the operators on the right-hand side of the above decomposition are Toeplitz operators. Thus, $M_{\Gamma}M_{\Psi}^*$ must be a Toeplitz operator as well.
This completes the proof.
\end{proof}
\begin{rem}\label{orem}
Let us show how Theorem \ref{prod_mains} can be used as an algorithm to decompose $M_{\Gamma} M_{\Psi}^*$ into the sum of elementary Toeplitz operators for any fixed $n \in \{1,\ldots,n\}$. Note that the main condition is the following.
\[
\big( \Gamma(\bm{\lambda}) - \Gamma_k(\bm{\lambda})\big) \big(  \Psi(\bm{\mu}) - \Psi_k(\bm{\mu}) \big)^* = 0
\tag{\textbf{A}}.
\]
Now, for any $i \in \{1,\ldots,n\}$, if we take $k=i$ and $\lambda_t=0=\mu_t$ for all $t \in \{1,\ldots,i-1\}$ , then 
\[
\big( \Gamma_{1,\ldots,i-1}(\bm{\lambda}) - \Gamma_{1,\ldots,i}(\bm{\lambda})\big) \big(  \Psi_{1,\ldots,i-1}(\bm{\mu}) - \Psi_{1,\ldots,i}(\bm{\mu}) \big)^* = 0,
\]
which gives the following decomposition.
\[
\Gamma_{1,\ldots,i-1} (\bm{\lambda}) \Psi_{1,\ldots,i-1} (\bm{\mu})^*  
=  \Gamma_{1,\ldots,i-1} (\bm{\lambda}) \Psi_{1,\ldots,i} (\bm{\mu})^* 
+ \Gamma_{1,\ldots,i} (\bm{\lambda})  \Psi_{1,\ldots,i-1}(\bm{\mu})^*
-\Gamma_{1,\ldots,i} (\bm{\lambda})  \Psi_{1,\ldots,i}(\bm{\mu})^*.
\]
Thus, we have got a recursive relation for decomposing the last term, which results in
\begin{align*}
&\Gamma(\bm{\lambda}) \Psi(\bm{\mu})^* \\
&= \Gamma(\bm{\lambda}) \Psi(0,\mu_2,\ldots,\mu_n)^* + 
 \Gamma(0,\lambda_2,\ldots,\lambda_n)  \Psi(\bm{\mu})^*\\
&-\Gamma(0,\lambda_2,\ldots,\lambda_n) \Psi(0,0, \mu_3,\ldots,\mu_n)^* 
- \Gamma(0,0, \lambda_3,\ldots,\lambda_n) \Psi(0,\mu_2,\ldots,\mu_n)^* \\
&+\Gamma(0,0, \lambda_3,\ldots,\lambda_n) \Psi(0,0,0,\mu_4,\ldots,\mu_n)^* 
+ \Gamma(0,0,0,\lambda_4,\ldots,\lambda_n) \Psi(0,0,\mu_3,\ldots,\mu_n)^* \\
&-\Gamma(0,0, 0, \lambda_4,\ldots,\lambda_n) \Psi(0,0,0,0,\mu_5,\ldots,\mu_n)^* 
- \Gamma(0,0,0,0,\lambda_5,\ldots,\lambda_n) \Psi(0,0,0,\mu_4,\ldots,\mu_n)^* \\
& \vdots\\
&(-1)^{n-2} \Gamma(0,\ldots,0,\lambda_{n-1}, \lambda_n) \Psi(0,\ldots,0, \mu_n)^* 
+ (-1)^{n-2} \Gamma(0,\ldots,0, \lambda_n) \Psi(0,\ldots,0,\mu_{n-1}, \mu_n)^* \\
&(-1)^{n-1} \Gamma(0,\ldots,0,\lambda_n) \Psi(0,\ldots,0)^* 
+ (-1)^{n-1} \Gamma(0,\ldots,0) \Psi(0,\ldots,0,\mu_n)^* + (-1)^n \Gamma \big(\bm{0} \big) \Psi \big(\bm{0} \big)^*.
\end{align*}
Note that the last three terms correspond to Toeplitz operators. 
Each of the rest of the terms can be again decomposed using condition $(A)$ by choosing the appropriate $k \in \{1,\ldots,n\}$ and $\lambda_i$'s and $\mu_j$'s to be zero.  For instance,  using condition $(\textbf{A})$, when $\mu_1=0$ and $k=2$, we get
\[
\big(\Gamma (\bm{\lambda})  - \Gamma_2 (\bm{\lambda})  \big) \cdot \big( \Psi_1( \bm{\mu}) - \Psi_{1,2}( \bm{\mu}) \big)^* =0,
\]
giving us the following decomposition
\[
\Gamma(\bm{\lambda}) \Psi_1(\bm{\mu})^* =  \Gamma(\bm{\lambda}) \Psi_{1,2}( \bm{\mu})^* + 
\Gamma_2 (\bm{\lambda})   \Psi_{1}(\bm{\mu})^* 
- \Gamma_2 (\bm{\lambda}) \Psi_{1,2}( \bm{\mu})^*.
\]
This will give a decomposition for the first term on the right-hand side of the big identity above.  The goal is to keep on decomposing the terms till we are left with symbols depending on different sets of variables. But it is evident that unless we have a fixed $n$, we cannot explicitly identify when the process will culminate, as each of the above terms will keep on having decompositions into more and more components. However, for any fixed choice of $n \in \Nat$, we can completely decompose $M_{\Gamma}M_{\Psi}^*$ in the same manner as we have seen for $n=1,2,3$.  Now, if someone is interested in what terms will be present in the decomposition of $M_{\Gamma}M_{\Psi}^*$, then based on the $n=1,2,3$ cases it is likely that modulo the sign, given any non-empty subset $A \subseteq \{1,\ldots,n\}$, we will have all the terms of the form $M_{\Gamma_A} M_{\Psi_B}^*$, where $B \subseteq \{1,\ldots,n\}$ such that $A \cup B = \{1,\ldots,n\}$.
\end{rem}
\section{Beurling-type Toeplitz ranges}\label{Beurling}
In the scalar cases, the main result in \cite{KPS} shows that a partially isometric Toeplitz operator $T_{\phi}$ on $H^2(\D^n)$ must be of the following form.
\[
T_{\phi} = M_{\phi_1}^*M_{\phi_2},
\]
where $\phi_1, \phi_2$ are inner functions on $\D^n$, but depending on disjoint set of variables. As a consequence of this factorization, one observes that 
\[
\ran T_{\phi} = M_{\phi_1} H^2(\D^n); \quad \ran T_{\phi}^* = M_{\phi_2} H^2(\D^n).
\]
In other words, the range of the partially isometric Toeplitz operator must be Beurling-type.  This section aims to prove that this result is true even for vector-valued Hardy spaces.  Since the commutativity of the symbols is lacking in vector-valued Hardy spaces, we have to pursue a completely new and different approach based on the characterization of Beurling-type invariant subspaces of $H_{\cle}^2(\D^n)$ via restriction operators. We begin by first showing that the range of partially isometric Toeplitz operators is a shift-invariant subspace of $H_{\cle}^2(\D^n)$.
\begin{propn}\label{shift-invariant}
If a Toeplitz operator $T_{\Phi}$ on $H_{\clb(\cle)}^2(\D^n)$ is a partial isometry, then the range of $T_{\Phi}$ is a $(M_{z_1}, \ldots, M_{z_n})$-joint invariant closed subspace of $H_{\cle}^2(\D^n)$.
\end{propn}
\begin{proof}
We already know that $T_{\Phi}$ is a partial isometry if and only if $T_{\Phi}^*$ is a partial isometry.  Thus,  $T_{\Phi}$ is a partial isometry implies that both $\ran T_{\Phi}$ and $\ran T_{\Phi}^*$ are closed subspaces of $H_{\cle}^2(\D^n)$.  Thus, we only need to show the $(M_{z_1}, \ldots, M_{z_n})$-joint invariance of $\ran T_{\Phi}$. To prove this, let us note that for any $f \in \ran T_{\Phi}$, we have the following inequalities for any arbitrary but fixed $i \in \{1,\ldots,n\}$.
\[
\|z_i f\| \geq \|T_{\Phi}^* M_{z_i} f\| \geq \|M_{z_i}^* T_{\Phi}^* M_{z_i} f\| = \|T_{\Phi}^* f\| = \|f\| = \|z_i f\|,
\]
and thus, $\|T_{\Phi}^* z_i f\| = \|z_i f \|$. Since $T_{\Phi}^*$ is a partial isometry, we must have $z_i f \in \ran T_{\Phi}$. Since $i$ was arbitrarily chosen, this must be true for all $i \in \{1,\ldots,n\}$. This completes the proof.
\end{proof}
Before proving our next result, let us highlight some useful facts involving shift-invariant subspaces of $H_{\cle}^2(\D^n)$. Given any $(M_{z_1}, \ldots, M_{z_n})$-joint invariant closed subspace $\cls$ of $H_{\cle}^2(\D^n)$, we can always associate the following restriction operators
\[
R_i := M_{z_i}|_{\cls},
\]
for each $i \in \{1,\ldots,n\}$. It is well known that these operators play a crucial role in characterizing shift-invariant subspaces of $H_{\cle}^2(\D^n)$.  We define the subspace $\cls$ to be \textit{doubly commuting} if 
\[
[R_i^*,R_j] = 0
\]
for all distinct $i,j \in \{1,\ldots,n\}$. From \cite{SSW}, we know that the following equivalence holds:\\
$(i)$~$\cls$ is doubly commuting,\\
$(ii)$~$\cls$ is a Beurling-type subspace of $H_{\cle}^2(\D^n)$, that is, by definition, there exists a Hilbert space $\clf$, and an inner function $\Theta \in H_{\clb(\cle)}^{\infty}(\D^n)$ such that
\[
\cls = M_{\Theta} H_{\cle}^2(\D^n).
\]
In the following result, we prove that the range of a partially isometric Toeplitz operator always admits the above description.
\begin{proof}[Proof of Theorem \ref{doubcom}]
Our approach is to show that if $T_{\Phi}$ is a partial isometry, then $\ran T_{\Phi}^*$ is a Beurling-type invariant subspace of $H_{\cle}^2(\D^n)$.  This will prove what we want because then we can use the following equivalence
\[
T_{\Phi} \text{ is a partial isometry } \Leftrightarrow T_{\Phi}^* \text{ is a partial isometry},
\]
to conclude that $\ran T_{\Phi}$ is Beurling-type. It is evident from the preceding discussion that we need to prove $\ran T_{\Phi}^*$ is doubly commuting.  In other words, for any distinct $i,j \in \{1,\ldots,n\}$, the following conditions should be satisfied.
\[
R_i ^* R_j = R_j R_i^*.
\]
Since $T_{\Phi}$ is a partial isometry, we can further deduce that
\[
R_i = M_{z_i}|_{\ran T_{\Phi}^*} = M_{z_i} T_{\Phi}^*T_{\Phi},
\]
for all $i \in \{1,\ldots,n\}$. Now, let us establish a few conditions essential for the sequel. From Proposition \ref{shift-invariant}, we know that $\ran T_{\Phi}$ is a $(M_{z_1}, \ldots, M_{z_n})$-joint invariant subspace, and therefore,
\[
T_{\Phi}T_{\Phi}^* M_{z_i} T_{\Phi} = M_{z_i} T_{\Phi}.
\]
Acting on the left by $M_{z_i}^*$ gives
\[
M_{z_i}^* T_{\Phi}T_{\Phi}^* M_{z_i} T_{\Phi} = T_{\Phi}.
\]
Using identity (\ref{pi}), we deduce that
\[
M_{z_i}^* T_{\Phi} (M_{z_i}M_{z_i}^* + P_{\ker M_{z_i}^*})T_{\Phi}^* M_{z_i} T_{\Phi} = T_{\Phi},
\]
and therefore, using the Toeplitz criterion in Theorem \ref{Toep_char}, we get
\[
T_{\Phi} + M_{z_i}^* T_{\Phi} P_{\ker M_{z_i}^*}T_{\Phi}^* M_{z_i} T_{\Phi} = T_{\Phi}.
\]
Hence,
\[
M_{z_i}^* T_{\Phi}P_{\ker M_{z_i}^*}T_{\Phi}^* M_{z_i} T_{\Phi}  = 0,
\]
which again implies that 
\[
T_{\Phi}^* M_{z_i}^* T_{\Phi}P_{\ker M_{z_i}^*}T_{\Phi}^* M_{z_i} T_{\Phi}  = 0,
\]
and therefore,
\[
P_{\ker M_{z_i}^*}T_{\Phi}^* M_{z_i} T_{\Phi}  = 0.
\]
In other words, since $P_{\ker M_{z_i}^*} = I_{H_{\cle}^2(\D^n)} - M_{z_i} M_{z_i}^*$, we get
\begin{equation}\label{partial_1}
M_{z_i} T_{\Phi}^* T_{\Phi} = T_{\Phi}^* M_{z_i} T_{\Phi}.
\end{equation}
Since $i$ was arbitrarily chosen, the above identity holds for all $i \in \{1,\ldots, n\}$. Now, if we act on the right of both sides by $T_{\Phi}^*$ and on the left of both sides by $T_{\Phi}$, we get
\[
 T_{\Phi} M_{z_i} T_{\Phi}^* T_{\Phi} T_{\Phi}^* = T_{\Phi}T_{\Phi}^* M_{z_i} T_{\Phi} T_{\Phi}^*.
\]
Using the $M_{z_i}$-invariance of $\ran T_{\Phi}$ we get
\begin{equation}\label{partial_2}
M_{z_i} T_{\Phi}  T_{\Phi}^* = T_{\Phi} M_{z_i} T_{\Phi}^*  \quad (i \in \{1,\ldots, n\}).
\end{equation}
and therefore, using identity (\ref{partial_1}), we get
\begin{align*}
R_i^* R_j = T_{\Phi}^* T_{\Phi} M_{z_i}^* M_{z_j}  T_{\Phi}^* T_{\Phi} = T_{\Phi}^* M_{z_i}^* T_{\Phi} T_{\Phi}^* M_{z_j} T_{\Phi}  &= T_{\Phi}^* M_{z_i}^* M_{z_j} T_{\Phi}  \\
&= T_{\Phi}^* M_{z_j}  M_{z_i}^*  T_{\Phi},
\end{align*}
and
\[
R_j R_i^*   = M_{z_j} T_{\Phi}^* T_{\Phi} M_{z_i}^* T_{\Phi}^* T_{\Phi} = M_{z_j} T_{\Phi}^* T_{\Phi} M_{z_i}^* .
\]
The last equality follows from the fact that $\ker T_{\Phi}^*$ is $(M_{z_1}^*, \ldots, M_{z_n}^*)$-joint invariant. Therefore, $\ran T_{\Phi}^*$ is a doubly commuting shift-invariant subspace if and only if 
\begin{equation}\label{c1}
T_{\Phi}^* M_{z_j}  M_{z_i}^*  T_{\Phi} = M_{z_j} T_{\Phi}^* T_{\Phi} M_{z_i}^*,
\end{equation}
for all distinct $i,j \in \{1,\ldots,n\}$. Using the identity (\ref{partial_1}), we can again observe that
\begin{align*}
&T_{\Phi}^* M_{z_j}  P_{\ker T_{\Phi}^*}M_{z_i}^*  T_{\Phi} \\
&= T_{\Phi}^* M_{z_j} M_{z_i}^*  T_{\Phi} - T_{\Phi}^* M_{z_j}  T_{\Phi} T_{\Phi}^* M_{z_i}^*  T_{\Phi} \\
&= T_{\Phi}^* M_{z_j} M_{z_i}^*  T_{\Phi} - M_{z_j}  T_{\Phi}^* T_{\Phi} T_{\Phi}^*T_{\Phi} M_{z_i}^*   \\
&= T_{\Phi}^* M_{z_j} M_{z_i}^*  T_{\Phi} - M_{z_j}  T_{\Phi}^* T_{\Phi} M_{z_i}^*.
\end{align*}
Therefore, condition (\ref{c1}) is again equivalent to $T_{\Phi}^* M_{z_j}  P_{\ker T_{\Phi}^*}M_{z_i}^*  T_{\Phi}=0$. We will now show that this identity holds. From the identities (\ref{partial_1}) and (\ref{partial_2}), we can get for any distinct $i,j \in \{1,\ldots,n\}$.
\begin{align*}
M_{z_i}^* T_{\Phi}  T_{\Phi}^*M_{z_i} M_{z_j}^* T_{\Phi}  T_{\Phi}^*M_{z_j} &= M_{z_i}^* T_{\Phi}  T_{\Phi}^* M_{z_j}^* M_{z_i} T_{\Phi}  T_{\Phi}^*M_{z_j} \\
&= M_{z_i}^* T_{\Phi}   M_{z_j}^* T_{\Phi}^* T_{\Phi} M_{z_i}  T_{\Phi}^*M_{z_j} \\
&= M_{z_i}^* T_{\Phi}   M_{z_j}^* M_{z_i}  T_{\Phi}^*M_{z_j} \\
&= M_{z_i}^* T_{\Phi}   M_{z_i}  M_{z_j}^*  T_{\Phi}^*M_{z_j} \\
&= T_{\Phi} T_{\Phi}^*.
\end{align*}
Therefore, for any distinct $i,j \in \{1,\ldots,n\}$, we get
\begin{equation}\label{partial3}
M_{z_i}^* T_{\Phi}  T_{\Phi}^*M_{z_i} M_{z_j}^* T_{\Phi}  T_{\Phi}^*M_{z_j}  =  T_{\Phi} T_{\Phi}^* = M_{z_j}^* T_{\Phi}  T_{\Phi}^*M_{z_j} M_{z_i}^* T_{\Phi}  T_{\Phi}^*M_{z_i}.
\end{equation}
From this, we can further deduce that
\begin{align*}
&M_{z_i}^* P_{\ker T_{\Phi}^*}M_{z_i} M_{z_j}^* P_{\ker T_{\Phi}^*} M_{z_j}  \\
&=  M_{z_i}^* M_{z_i} M_{z_j}^* P_{\ker T_{\Phi}^*} M_{z_j}  - M_{z_i}^* T_{\Phi} T_{\Phi}^* M_{z_i} M_{z_j}^* P_{\ker T_{\Phi}^*} M_{z_j}  \\
&=  M_{z_j}^* P_{\ker T_{\Phi}^*} M_{z_j}  - M_{z_i}^* T_{\Phi} T_{\Phi}^* M_{z_i} M_{z_j}^* P_{\ker T_{\Phi}^*} M_{z_j}  \\
&=  M_{z_j}^* P_{\ker T_{\Phi}^*} M_{z_j}  - M_{z_i}^* T_{\Phi} T_{\Phi}^* M_{z_i} M_{z_j}^*  M_{z_j}\\
&+ M_{z_i}^* T_{\Phi} T_{\Phi}^* M_{z_i} M_{z_j}^* T_{\Phi} T_{\Phi}^* M_{z_j}  \\
&=  M_{z_j}^* P_{\ker T_{\Phi}^*} M_{z_j}  - M_{z_i}^* T_{\Phi} T_{\Phi}^* M_{z_i} + T_{\Phi} T_{\Phi}^*.
\end{align*}
Since the last equality on the right-hand side is self-adjoint, the above set of identities implies that for any distinct $i,j \in \{1,\ldots,n\}$, we get
\begin{equation}\label{partial4}
M_{z_i}^* P_{\ker T_{\Phi}^*}M_{z_i} M_{z_j}^* P_{\ker T_{\Phi}^*} M_{z_j}  =  M_{z_j}^* P_{\ker T_{\Phi}^*} M_{z_j}  M_{z_i}^* P_{\ker T_{\Phi}^*}M_{z_i}.
\end{equation}
Now let us compute $C_i ^* C_i$, where $C_i = P_{\ker T_{\Phi}^*}M_{z_i}|_{\ker T_{\Phi}^*}$ for $i \in \{1,\ldots,n\}$.
\begin{align*}
C_i ^* C_i = P_{\ker T_{\Phi}^*} M_{z_i}^* P_{\ker T_{\Phi}^*}M_{z_i}P_{\ker T_{\Phi}^*}  &= M_{z_i}^* P_{\ker T_{\Phi}^*}M_{z_i}P_{\ker T_{\Phi}^*} \\
&= M_{z_i}^* P_{\ker T_{\Phi}^*}M_{z_i}.
\end{align*}
The second and last equalities uses the fact that $\ran T_{\Phi}$ is $(M_{z_1},\ldots, M_{z_n})$-joint invariant. Now, let us observe that
\begin{align*}
&I_{\ker T_{\Phi}^*} - C_i ^* C_i \\
&=I_{\ker T_{\Phi}^*} -  M_{z_i}^* P_{\ker T_{\Phi}^*}M_{z_i} \\
&= I_{H_{\cle}^2(\D^n)}  - T_{\Phi}T_{\Phi}^* - M_{z_i}^* P_{\ker T_{\Phi}^*}M_{z_i} \\
&= I_{H_{\cle}^2(\D^n)}  - T_{\Phi}T_{\Phi}^* - I_{H_{\cle}^2(\D^n)} + M_{z_i}^* T_{\Phi}T_{\Phi}^*M_{z_i} \\
&=  M_{z_i}^* T_{\Phi}T_{\Phi}^*M_{z_i}   - T_{\Phi}T_{\Phi}^*.
\end{align*}
By compressing the above identity with respect to $P_{\ker T_{\Phi}^*}$, we get
\[
I_{\ker T_{\Phi}^*} - C_i ^* C_i = P_{\ker T_{\Phi}^*} M_{z_i}^* T_{\Phi}T_{\Phi}^*M_{z_i} P_{\ker T_{\Phi}^*},
\]
Therefore, for each $i \in \{1,\ldots,n\}$, there exists isometries
\[
Y_i: \overline{(I_{\ker T_{\Phi}^*} - C_i^* C_i)^{\frac{1}{2}} \clh} \raro \overline{T_{\Phi}^*M_{z_i} P_{\ker T_{\Phi}^*} \clh}
\]
such that 
\[
Y_i (I_{\ker T_{\Phi}^*} - C_i^* C_i)^{\frac{1}{2}} = T_{\Phi}^*M_{z_i} P_{\ker T_{\Phi}^*}.
\]
Now using conditions (\ref{partial_1}), (\ref{partial_2}), and (\ref{partial3}), we deduce that
\begin{align*}
&(I_{\ker T_{\Phi}^*} - C_i^* C_i)  (I_{\ker T_{\Phi}^*} - C_j^* C_j) \\
&= \big(M_{z_i}^* T_{\Phi}T_{\Phi}^*M_{z_i}   - T_{\Phi}T_{\Phi}^* \big) \big(M_{z_j}^* T_{\Phi}T_{\Phi}^*M_{z_j}   - T_{\Phi}T_{\Phi}^* \big)\\
&= M_{z_i}^* T_{\Phi}T_{\Phi}^*M_{z_i} M_{z_j}^* T_{\Phi}T_{\Phi}^*M_{z_j}   - M_{z_i}^* T_{\Phi}T_{\Phi}^*M_{z_i} T_{\Phi}T_{\Phi}^*  \\
&- T_{\Phi}T_{\Phi}^*  M_{z_j}^* T_{\Phi}T_{\Phi}^*M_{z_j}  +  T_{\Phi}T_{\Phi}^* \\
&=  T_{\Phi}T_{\Phi}^*  - M_{z_i}^* T_{\Phi}M_{z_i} T_{\Phi}^* T_{\Phi}T_{\Phi}^*  \\
&- T_{\Phi}  M_{z_j}^* T_{\Phi}^* T_{\Phi}T_{\Phi}^*M_{z_j}  +  T_{\Phi}T_{\Phi}^* \\
&=  T_{\Phi}T_{\Phi}^*  -  T_{\Phi} T_{\Phi}^* - T_{\Phi}  T_{\Phi}^* +  T_{\Phi}T_{\Phi}^* \\
&=0.
\end{align*}
Since both $(I_{\ker T_{\Phi}^*} - C_i^* C_i)$ and $(I_{\ker T_{\Phi}^*} - C_j^* C_j)$ are non-negative operators, therefore, we have 
\[
(I_{\ker T_{\Phi}^*} - C_i^* C_i)^{\frac{1}{2}}  (I_{\ker T_{\Phi}^*} - C_j^* C_j)^{\frac{1}{2}} = 0.
\]
This implies that 
\[
T_{\Phi}^*M_{z_j} P_{\ker T_{\Phi}^*} M_{z_i}^* T_{\Phi} = Y_j (I_{\ker T_{\Phi}^*} - C_j^* C_j)^{\frac{1}{2}} (I_{\ker T_{\Phi}^*} - C_i^* C_i)^{\frac{1}{2}} Y_i^* = 0.
\]
This completes the proof.
\end{proof}

\section{Partially isometric Toeplitz operators}\label{partial}
This section will prove that partially isometric Toeplitz operators always admit a factorization into Toeplitz operators corresponding to inner symbols.
\begin{proof}[Proof of Theorem \ref{mainthm1}]
From Theorem \ref{doubcom}, we know that $T_{\Phi}$ is a partial isometry implies that both $\ran T_{\Phi}$ and $\ran T_{\phi}^*$ are Beurling-type invariant subspace of $H_{\cle}^2(\D^n)$.  Thus, there exist Hilbert spaces $\clf, \clg$ and inner functions $\Gamma(\bm{z}) \in H_{\clb(\clf, \cle)}^{\infty}(\D^n)$, $\Psi(\bm{z}) \in H_{\clb(\clg, \cle)}^{\infty}(\D^n)$ such that 
\[
\ran T_{\Phi} = M_{\Gamma} H_{\clf}^2(\D^n),
\]
and 
\[
\ran T_{\Phi}^* = M_{\Psi} H_{\clg}^2(\D^n).
\]
From the above identities, it follows that
\[
T_{\Phi} T_{\Phi}^* = M_{\Gamma}M_{\Gamma}^*; \quad T_{\Phi}^* T_{\Phi} = M_{\Psi}M_{\Psi}^*.
\]
Since $T_{\Phi}$ is partial isometry, it further implies that 
\[
M_{\Psi}M_{\Psi}^* = T_{\Phi}^*T_{\Phi} = T_{\Phi}^* T_{\Phi} T_{\Phi}^* T_{\Phi} = T_{\Phi}^* M_{\Gamma}M_{\Gamma}^* T_{\Phi},
\]
and therefore, $\|M_{\Psi}^* h\| = \|M_{\Gamma}^* T_{\Phi} h \|$ for any $h \in H_{\cle}^2(\D^n)$. Since $M_{\Psi}$ is an isometry, we have $\ran M_{\Psi}^* = H_{\clg}^2(\D^n)$ and therefore, we can define an isometry $X: H_{\clg}^2(\D^n) \raro H_{\clf}^2(\D^n)$ by 
\[
X M_{\Psi}^* h = M_{\Gamma}^* T_{\Phi} h,
\]
in other words, $X M_{\Psi}^* = M_{\Gamma}^* T_{\Phi}$ on $H_{\cle}^2(\D^n)$.  Using this map, we can observe that
\[
T_{\Phi} = M_{\Gamma}  M_{\Gamma}^* T_{\Phi} = M_{\Gamma}  X M_{\Psi}^*.
\]
This further implies that
\[
X = M_{\Gamma}^* T_{\Phi} M_{\Psi},
\]
and hence, 
\[
M_{z_i}^* X M_{z_i} = M_{z_i}^* M_{\Gamma}^* T_{\Phi} M_{\Psi} M_{z_i} =  M_{\Gamma}^* M_{z_i}^* T_{\Phi} M_{z_i} M_{\Psi}= M_{\Gamma}^* T_{\Phi} M_{\Psi}  = X,
\]
for all $i \in \{1,\ldots,n\}$. Moreover, 
\begin{align*}
XX^* = M_{\Gamma}^* T_{\Phi} M_{\Psi} M_{\Psi}^* T_{\Phi}^* M_{\Gamma}  = M_{\Gamma}^* T_{\Phi} T_{\Phi}^* T_{\Phi} T_{\Phi}^* M_{\Gamma} &= M_{\Gamma}^* T_{\Phi} T_{\Phi}^* M_{\Gamma} \\
&= M_{\Gamma}^* M_{\Gamma} M_{\Gamma}^* M_{\Gamma} \\
&= I_{H_{\clf}^2(\D^n)},
\end{align*}
shows that $X$ is a unitary between $H_{\clg}^2(\D^n)$ and $H_{\clf}^2(\D^n)$. Again, for any $i \in \{1,\ldots,n\}$, we get
\[
I_{H_{\clf}^2(\D^n)} = X X^* = M_{z_i}^* X M_{z_i}M_{z_i}^*X^*M_{z_i} =I_{H_{\clf}^2(\D^n)} - M_{z_i}^* X P_{\ker M_{z_i}^*} X^*M_{z_i} .
\]
Thus, $M_{z_i}^* X P_{\ker M_{z_i}^*} X^*M_{z_i} = 0$, and hence $(I_{H_{\clf}^2(\D^n)}- M_{z_i} M_{z_i}^*) X^* M_{z_i} = 0$ and therefore, 
\begin{equation}\label{point1}
X^* M_{z_i}  = M_{z_i} X^*,
\end{equation}
for all $i \in \{1,\ldots,n\}$.  Similarly, for any $i \in \{1,\ldots,n\}$ we have
\[
 I_{H_{\clg}^2(\D^n)} = X^* X = M_{z_i}^* X^* M_{z_i}M_{z_i}^*XM_{z_i} =I_{H_{\clg}^2(\D^n)} - M_{z_i}^* X^* P_{\ker M_{z_i}^*} X M_{z_i}.
\]
This implies that $M_{z_i}^* X^* P_{\ker M_{z_i}^*} X M_{z_i} = 0$ and thus, we get
\begin{equation}\label{point2}
X M_{z_i}  = M_{z_i} X,
\end{equation}
for all $i \in \{1,\ldots,n\}$. From conditions (\ref{point1}) and (\ref{point2}), it implies that $X$ must be a constant unitary from $H_{\clg}^2(\D)$ to  $H_{\clf}^2(\D)$. Therefore, we can re-write the map $X$ as $I_{H^2(\D^n)} \otimes X$, where $X: \clg \raro \clf$ is a unitary. Based on these observations, we can now write
\[
T_{\Phi} = M_{\Gamma}  (I_{H_{\cle}^2(\D^n)} \otimes X) M_{\Psi}^* = M_{\Gamma} M_{\tilde{\Psi}} ^*,
\]
where $\tilde{\Psi}(\bm{z}): = \Psi(\bm{z}) X^* \in H_{\clb(\clf,\cle)}^{\infty}(\D^n)$ is an inner function. Since $T_{\Phi} = M_{\Gamma} M_{\tilde{\Psi}} ^*$ is a Toeplitz operator,  using Theorem \ref{prod_mains}, we get that
\[
\big( \Gamma(\bm{\lambda}) - \Gamma_k(\bm{\lambda}) \big) \big( \tilde{\Psi}(\bm{\mu}) - \tilde{\Psi}_k(\bm{\mu})\big)^* = 0,
\]
for all $\bm{\lambda}, \bm{\mu} \in \D^n$ and $k \in \{1,\ldots,n\}$.
Conversely, if the Toeplitz operator admits a factorization like $T_{\Phi} = M_{\Gamma} M_{\Psi}^*$. Then 
\[
T_{\Phi} T_{\Phi}^* = M_{\Gamma} M_{\Psi}^* M_{\Psi} M_{\Gamma}^* = M_{\Gamma}M_{\Gamma}^*,
\]
shows that $T_{\Phi}$ is a partial isometry. This completes the proof.
\end{proof}
Using Theorem \ref{simple}, Theorem \ref{bidisc} and Theorem \ref{tridisc}, we have the following finer results.
\begin{thm}\label{psimple}
Let $T_{\Phi}$ be a non-constant Toeplitz operator on $H_{\cle}^2(\D)$. Then the following statements are equivalent.\\
$(i)$~$T_{\Phi}$ is a partial isometry,\\
$(ii)$~there exists a Hilbert space $\clf$,  and inner functions $\Gamma(z), \Psi(z) \in H_{\clb(\clf, \cle)}^{\infty}(\D)$ such that
\[
T_{\Phi} = M_{\Gamma}M_{\Psi}^*  
=  \Gamma(0) M_{\Psi}^* + M_{\Gamma} \Psi(0)^* - \Gamma(0) \Psi(0)^*.
\]
\end{thm}
In the case of the bidisc, we have the following result.
\begin{thm}\label{pbidisc}
Let $T_{\Phi}$ be a non-constant Toeplitz operator on $H_{\cle}^2(\D^2)$. Then the following statements are equivalent.\\
$(i)$~$T_{\Phi}$ is a partial isometry,\\
$(ii)$~there exists a Hilbert space $\clf$,  and inner functions $\Gamma(\bm{z}), \Psi(\bm{z}) \in H_{\clb(\clf, \cle)}^{\infty}(\D^2)$ such that
\begin{align*}
T_{\Phi} &= M_{\Gamma}M_{\Psi}^*  \\
&=  M_{\Gamma} \Psi(\bm{0})^*\\
&+M_{\Gamma_1} M_{\Psi_2}^* - M_{\Gamma_1} \Psi(\bm{0})^*\\
&+M_{\Gamma_2} M_{\Psi_1}^*  - M_{\Gamma_2} \Psi(\bm{0})^* \\
& + \Gamma(\bm{0}) M_{\Psi}^* + \Gamma(\bm{0}) M_{\Psi_1}^* + \Gamma(\bm{0}) M_{\Psi_2}^*- \Gamma(\bm{0}) \Psi(\bm{0})^*.
\end{align*}
\end{thm}
For the tridisc, we obtain the following characterization.
\begin{thm}\label{ptridisc}
Let $T_{\Phi}$ be a non-constant Toeplitz operator on $H_{\cle}^2(\D^2)$. Then the following statements are equivalent.\\
$(i)$~$T_{\Phi}$ is a partial isometry,\\
$(ii)$~there exists a Hilbert space $\clf$,  and inner functions $\Gamma(\bm{z}), \Psi(\bm{z}) \in H_{\clb(\clf, \cle)}^{\infty}(\D^3)$ such that
\begin{align*}
T_{\Phi} &= M_{\Gamma}M_{\Psi}^*  \\
&=  M_{\Gamma} \Psi(\bm{0})^*\\
&-M_{\Gamma_1} \Psi(\bm{0})^* + M_{\Gamma_1} M_{\Psi_{2,3}}^*\\
&-M_{\Gamma_2} \Psi(\bm{0})^* + M_{\Gamma_2} M_{\Psi_{1,3}}^*\\
&-M_{\Gamma_3} \Psi(\bm{0})^* + M_{\Gamma_3} M_{\Psi_{1,2}}^*\\
&+M_{\Gamma_{1,2}} \Psi(\bm{0})^*  - M_{\Gamma_{1,2}} M_{\Psi_{1,3}}^* -  M_{\Gamma_{1,2}} M_{\Psi_{2,3}}^* + M_{\Gamma_{1,2}} M_{\Psi_{3}}^*\\
&+M_{\Gamma_{1,3}} \Psi(\bm{0})^*  - M_{\Gamma_{1,3}} M_{\Psi_{1,2}}^* -  M_{\Gamma_{1,3}} M_{\Psi_{2,3}}^* + M_{\Gamma_{1,3}} M_{\Psi_{2}}^*\\
&+M_{\Gamma_{2,3}} \Psi(\bm{0})^*  - M_{\Gamma_{2,3}} M_{\Psi_{1,2}}^* - M_{\Gamma_{2,3}} M_{\Psi_{1,3}}^* + M_{\Gamma_{2,3}} M_{\Psi_1}^* \\
&+\Gamma(\bm{0}) M_{\Psi}^*  - \Gamma(\bm{0}) M_{\Psi_{1}}^* -  \Gamma(\bm{0}) M_{\Psi_{2}}^* - \Gamma(\bm{0}) M_{\Psi_{3}}^*\\
&+\Gamma(\bm{0}) M_{\Psi_{1,2}}^*  - \Gamma(\bm{0}) M_{\Psi_{1,3}}^* -  \Gamma(\bm{0}) M_{\Psi_{2,3}}^* - \Gamma(\bm{0}) \Psi(\bm{0})^*.
\end{align*}
\end{thm}
In a similar manner, any partially isometric Toeplitz operator on $H_{\cle}^2(\D^n)$ for any fixed $n>3$ can be characterized using Theorem \ref{prod_mains} and the algorithm described in Remark \ref{orem}.  

\begin{rem}
Following the same method as in Theorem \ref{Toep_char}, we can show that $X \in \clb(H_{\cle}^2(\D^n), H_{\clf}^2(\D^n))$ is a Toeplitz operator if and only if $(M_{z_i} \otimes I_{\clf})^* X (M_{z_i} \otimes I_{\cle}) = X$ for all $i \in \{1,\ldots,n\}$, where $\cle, \clf$ are different Hilbert spaces. Using this algebraic condition, we can follow the proof of Theorem \ref{mainthm1}, to obtain a similar characterization for partiallly isometric $T_{\Phi} \in \clb(H_{\cle}^2(\D^n), H_{\clf}^2(\D^n))$. Our results can be directly used when $\mbox{dim } \cle = \mbox{dim } \clf$. In particular, we can always construct a unitary $U: \cle \raro \clf$, which can be further extended to an unitary $I \otimes U: H_{\cle}^2(\D^n) \raro H_{\clf}^2(\D^n)$. So, if $T_{\Phi} \in \clb(H_{\cle}^2(\D^n), H_{\clf}^2(\D^n))$ is a partial isometry then $U^* T_{\Phi} \in \clb(H_{\cle}^2(\D^n), H_{\clf}^2(\D^n))$ is a partially isometric Toeplitz operator as well. Using Theorem \ref{mainthm1}, we will get $U^* T_{\Phi} = M_{\Gamma} M_{\Psi}^*$, and hence, $T_{\Phi} = M_{U \Gamma} M_{\Psi}^*$.
\end{rem}

Now, an immediate question that arises is what effect does a partially isometric Toeplitz operator have on its symbol? We have the following complete answer.
\begin{cor}\label{symb}
If $T_{\Phi}$ is a partially isometric Toeplitz operator on $H_{\cle}^2(\D^n)$, then $\Phi(\bm{z})$ is a partial isometry a.e. on $\mathbb{T}^n$.
\end{cor}
\begin{proof}
From Theorem \ref{mainthm1} and Theorem \ref{ptoep}, we get 
\[
T_{\Phi} = M_{\Gamma} M_{\Psi}^* = T_{\Gamma \Psi^*}.
\]
This implies that $\Phi(\bm{z}) = \Gamma(\bm{z}) \Psi(\bm{z})^*$ a.e. on $\mathbb{T}^n$. Since $\Gamma(\bm{z}), \Psi(\bm{z})$ are inner functions it implies that $\Gamma(\bm{z)} \Gamma(\bm{z)}^*$ is projection-valued and $\Psi(\bm{z})^* \Psi(\bm{z}) = I_{\clf}$ a.e. on $\mathbb{T}^n$ and therefore, 
\[
\Phi(\bm{z})\Phi(\bm{z})^* = \Gamma(\bm{z}) \Psi^*(\bm{z}) \Psi(\bm{z}) \Gamma^*(\bm{z}) = \Gamma(\bm{z}) \Gamma^*(\bm{z})  \quad (\text{a.e.  on } \mathbb{T}^n).
\]
This shows that $\Phi(\bm{z})$ is a partial isometry, a.e. on $\mathbb{T}^n$. This completes the proof.
\end{proof}
\begin{rem}
The converse direction is not true, and we present an example here. Let $\theta(z) = \frac{z}{2}$ for all $z \in \D$, and let us consider the symbol $\Phi \in L_{\clb(\mathbb{C}^2)}^{\infty}(\mathbb{T})$ defined by 
\[
\Phi(e^{it}) := \begin{bmatrix}
\theta(e^{it}) & (1 - |\theta(e^{it})|^2)^{\frac{1}{2}}\\
0 & 0
\end{bmatrix}  =  \begin{bmatrix}
\theta(e^{it}) & \frac{\sqrt{3}}{2}\\
0 & 0
\end{bmatrix}\in \clb(\mathbb{C}^2).
\]
It is clear that 
\[
\Phi(e^{it})^* \Phi(e^{it}) = \begin{bmatrix}
1 & 0\\
0 & 0
\end{bmatrix}.
\]
Thus, $\Phi(e^{it})$ is a partial isometry on $\mathbb{T}$. However, the corresponding Toeplitz operator
\[
T_{\Phi} = \begin{bmatrix}
T_{\theta} & T_{\frac{\sqrt{3}}{2}}\\
0 & 0
\end{bmatrix} \in H_{\mathbb{C}^2}^2(\D),
\]
is not partially isometric. One can easily see that if we want $T_{\Phi} T_{\Phi}^* T_{\Phi} = T_{\Phi}$, then a necessary condition is $T_{\theta} T_{\theta}^* = T_{\frac{1}{4}}$. This is not possible as $\theta$ is an analytic function.
\end{rem}
We will now characterize partially isometric Toeplitz operators with analytic symbols.
\begin{thm}\label{main_analytic}
Let $\Phi \in H_{\clb(\cle)}^{\infty}(\D^n)$. Then $M_{\Phi}$ is a partial isometry on $H_{\cle}^2(\D^n)$ if and only if there exists a Hilbert space $\clf$, an inner function $ \Gamma \in H_{\clb(\cle)}^{\infty}(\D^n)$, and an isometry $V: \clf \raro \cle$ such that 
\[
\Phi(\bm{z}) = \Gamma(\bm{z}) V^* \quad (\bm{z} \in \D^n).
\]
\end{thm}
\begin{proof}
If $M_{\Phi}$ is a partial isometry, then by Theorem \ref{mainthm1}, we will get the following factorization
\[
M_{\Phi} = M_{\Gamma}M_{\Psi}^*,
\]
for some Hilbert space $\clf$, and inner functions $\Gamma(\bm{z}), \Psi(\bm{z}) \in H_{\clb(\clf,\cle)}^{\infty}(\D^n)$. Since $\Phi$ is a bounded analytic function, we know
\[
M_{z_i} M_{\Phi} = M_{\Phi} M_{z_i},
\]
which gives,
\[
M_{z_i} M_{\Gamma}  M_{\Psi}^* = M_{\Gamma} M_{\Psi}^* M_{z_i},
\]
and therefore,
\[
M_{\Gamma} M_{z_i} M_{\Psi}^* = M_{\Gamma} M_{\Psi}^* M_{z_i}.
\]
Using the fact that $M_{\Gamma}$ is an isometry we get
\[
M_{z_i} M_{\Psi}^* =  M_{\Psi}^* M_{z_i}.
\]
This implies that $M_{\Psi}$ must be a constant isometry from $H_{\clf}^2(\D)$ to $H_{\cle}^2(\D)$.  Thus, $M_{\Psi} = I_{H^2(\D^n)} \otimes V$ for some isometry $V: \clf \raro \cle$, and using this we get
\[
M_{\Phi} = M_{\Gamma} (I_{H^2(\D^n)} \otimes V^*) = M_{\Gamma V^*}.
\]
Conversely, if the bounded analytic function admits the factorization $\Phi = \Gamma V^*$ for some inner function $\Gamma \in H_{\clb(\cle)}^{\infty}(\D^n)$, and some isometry $V:\clf \raro \cle$, then 
\[
T_{\Phi}T_{\Phi}^* = M_{\Gamma} (I_{H^2(\D^n)} \otimes V^*) (I_{H^2(\D^n)} \otimes V)M_{\Gamma}^* = M_{\Gamma} M_{\Gamma}^*,
\]
shows that $T_{\Phi}$ is a partial isometry. This completes the proof.
\end{proof}
We will now characterize partially isometric Toeplitz operators, which are hyponormal.  Let us first recall that a bounded operator $T$ on a Hilbert space $\clh$ is said to be hyponormal if $T^*T \geq TT^*$. We refer to the monograph \cite{MP} for an elaborate discussion on these operators. In the context of hyponormal Toeplitz operators, there has been important development starting with the celebrated result of Cowen \cite{Cowen}. In recent times, several authors have contributed to its extension to the case of block Toeplitz operators, for instance, Gu et al. \cite{GHR} and Curto et al. \cite{CHL4}. One of the criteria for a block Toeplitz operator $T_{\Phi}$ on $H_{\mathbb{C}^m}^2(\D)$ to be hyponormal is that the symbol $\Phi$ should be normal a.e. on $\mathbb{T}$ \cite[Theorem 3.3]{GHR}.  The proof shows that this feature is due to the finite dimensionality of the underlying Hilbert space in $H_{\mathbb{C}^m}^2(\D)$. The following result shows that this feature may not hold if we consider hyponormal Toeplitz operators on $H_{\cle}^2(\D^n)$, where $\cle$ can be infinite-dimensional. Before going into the proof, let us highlight that the following result is an extension of \cite[Corollary 5.1]{KPS}, and the initial part of the proof follows from their method.
\begin{thm}\label{hyponormal}
Let $T_{\Phi}$ be a partially isometric Toeplitz operator on $H_{\cle}^2(\D^n)$. Then, the following statements are equivalent.\\
$(i)$~$T_{\Phi}$ is hyponormal,\\
$(ii)$~there exists Hilbert spaces $\clf, \clg$ and inner functions $\Psi \in H_{\clb(\clf,\cle)}^{\infty}(\D^n)$, $ \Theta \in H_{\clb(\clf)}^{\infty}(\D^n)$ such that
\[
T_{\Phi} = M_{\Psi} M_{\Theta} M_{\Psi}^*.
\]
\end{thm}
\begin{proof}
$T_{\Phi}$ is a partial isometry will imply that $T_{\Phi} = M_{\Gamma} M_{\Psi}^*$. Using this identity, we get,
\[
T_{\Phi}^* T_{\Phi} - T_{\Phi} T_{\Phi}^* = M_{\Psi} M_{\Psi}^* - M_{\Gamma}M_{\Gamma}^*.
\]
Now, hyponormality of $T_{\Phi}$ implies that 
\[
 M_{\Gamma}M_{\Gamma}^* \leq M_{\Psi} M_{\Psi}^*.
\]
By Douglas's lemma \cite{Douglas2}, there exists a contraction $Z: H_{\clf}^2(\D^n) \raro H_{\clf}^2(\D^n)$ such that
\[
 M_{\Gamma} = M_{\Psi} Z.
\]
This implies that for any $i \in \{1,\ldots,n\}$
\[
M_{\Psi} M_{z_i} Z =  M_{z_i}  M_{\Psi} Z = M_{z_i} M_{\Gamma} = M_{\Gamma} M_{z_i} = M_{\Psi} Z M_{z_i},
\]
Since $M_{\Psi}$ is an isometry, the above identity further implies that $Z M_{z_i} = M_{z_i} Z$. Thus, there exists $\Theta \in H_{\clb(\clf)}^{\infty}(\D^n)$ such that
\[
Z = M_{\Theta},
\]
Furthermore, $\Theta$ is an inner function because so are $\Gamma$ and $\Psi$. Hence,
\[
T_{\Phi} = M_{\Gamma} M_{\Psi}^* = M_{\Psi} M_{\Theta} M_{\Psi}^*.
\]
Conversely, if $T_{\Phi} = M_{\Psi} M_{\Theta} M_{\Psi}^*$, then 
\[
T_{\Phi}T_{\Phi}^* = M_{\Psi} M_{\Theta} M_{\Psi}^*M_{\Psi} M_{\Theta}^* M_{\Psi}^* = M_{\Psi} M_{\Theta} M_{\Theta}^* M_{\Psi}^* \leq M_{\Psi}M_{\Psi}^* = T_{\Phi}^* T_{\Phi},
\]
implies that $T_{\Phi}$ is hyponormal. This completes the proof.
\end{proof}
Normal operators on Hilbert spaces, that is, operators $T$ on $\clh$ satisfying $[T^*, T]=0$, are important cases of hyponormal operators. Let us end this section by characterizing partially isometric and normal Toeplitz operators.
\begin{cor}\label{normal}
Let $T_{\Phi}$ be a partially isometric Toeplitz operator on $H_{\cle}^2(\D^n)$. Then, the following statements are equivalent.\\
$(i)$~$T_{\Phi}$ is normal,\\
$(ii)$~there exists a Hilbert space $\clf$, an inner function $\Psi \in H_{\clb(\clf,\cle)}^{\infty}(\D^n)$, and a unitary $U \in \clb(\clf)$ such that
\[
T_{\Phi} = M_{\Psi} U M_{\Psi}^*.
\]
\end{cor}
\begin{proof}
First let us note that if $T_{\Phi} = M_{\Psi} U M_{\Psi}^*$, then
\[
T_{\Phi}^* T_{\Phi}  = M_{\Psi} M_{\Psi}^* = T_{\Phi} T_{\Phi}^*. 
\]
Hence, $T_{\Phi}$ is both normal as well as a partial isometry.  Next, suppose that the partially isometric $T_{\Phi}$ is also normal. Since a normal operator is also hyponormal, we can use the above Theorem \ref{hyponormal} to claim that there exist Hilbert spaces $\clf, \clg$ and inner functions $\Psi \in H_{\clb(\clf,\cle)}^{\infty}(\D^n)$, $ \Theta \in H_{\clb(\clf)}^{\infty}(\D^n)$ such that
\[
T_{\Phi} = M_{\Psi} M_{\Theta} M_{\Psi}^*.
\]
Now, $T_{\Phi}$ being normal further implies that
\[
M_{\Psi} M_{\Theta}^* M_{\Psi}^* M_{\Psi} M_{\Theta} M_{\Psi}^* = M_{\Psi} M_{\Theta} M_{\Psi}^*M_{\Psi} M_{\Theta}^* M_{\Psi}^*,
\]
and hence,
\[
M_{\Psi}  M_{\Psi}^* = M_{\Psi} M_{\Theta} M_{\Theta}^* M_{\Psi}^*.
\]
This further implies that
\[
M_{\Psi} (I_{H_{\cle}^2(\D^n)} - M_{\Theta} M_{\Theta}^*) M_{\Psi}^* = 0,
\]
and therefore, if we act on the above identity by $M_{\Psi}^*$ on the left side and by $M_{\Psi}$ on the right side, then we get
\[
I_{H_{\cle}^2(\D^n)} - M_{\Theta} M_{\Theta}^* = 0.
\]
Since $M_{\Theta}$ was an isometry to begin with, the above identity further implies that $M_{\Theta}$ is a unitary operator. Using Corollary \ref{unitary}, we can deduce that $M_{\Theta}$ must be a constant unitary, in other words, $M_{\Theta} = I_{H^2(\D^n)} \otimes U$, where $U: \clf \raro \clf$ is a unitary. Hence, the Toeplitz operator $T_{\Phi}$ becomes
\[
T_{\Phi} = M_{\Psi} U M_{\Psi}^*.
\]
This completes the proof.
\end{proof}
\section{Vectorial Toeplitz operators on  Hilbert spaces}\label{abstract}
An important aspect of working with Toeplitz operators on vector-valued Hardy spaces is that they give important information about general operators on abstract Hilbert spaces. In particular, we can use our main result in Section 3 to prove a characterization for vectorial Toeplitz operators. Throughout this section, we assume $\clh$ to be a separable infinite-dimensional Hilbert space over complex numbers. To begin with, let us recall a few definitions essential for the sequel.

\begin{defn}
An operator $S$ on $\clh$ is a \textit{shift} if $S$ is an isometry and $S^{*n} \raro 0$ in the strong operator topology as $n \raro \infty$.
\end{defn}

\begin{defn}
A tuple of commuting shift operators $(S_1,\ldots,S_n)$ on $\clh$ is \textit{doubly commuting} if $S_i^* S_j = S_j S_i^*$ for all distinct $i,j \in \{1,\ldots,n\}$. 
\end{defn}
For a tuple of doubly commuting shift operators, $S$ on $\clh$, there exists a subspace of utmost importance, namely the wandering subspace denoted by 
\[
\clw(S):= \underset{k=1}{\overset{n}{\cap}} \ker S_k^*.
\]
This subspace generates the Hilbert space in the following manner.
\[
\clh = \bigoplus_{\bm{k} \in \Nat^n} S^{\bm{k}} \clw(S),
\]
for instance, see \cite[Corollary 3.2]{JS} and originally \cite{Halmos} for a single shift operator. Using this structure of $\clh$, we can define a unitary $U:\clh \raro H_{\clw(S)}^2(\D^n)$ by the following action.
\[
U (S^{\bm{k}} \eta) = M_{\bm{z}}^{\bm{k}} \eta,
\]
for all $\bm{k} \in \Nat^n$ and $\eta \in \clw(S)$. The construction of the unitary further facilitates the following intertwining relations
\begin{equation}\label{inter}
M_{z_k} U = U S_k; \quad M_{z_k}^* U = U S_k^* \quad (k \in \{1,\ldots,n\}).
\end{equation}
All the above facts will be used in the sequel.  Motivated by the definition set by Page in \cite{Page}, we define vectorial Toeplitz operators corresponding to a tuple of shifts in the following manner.
\begin{defn}
Given a $n$-tuple of doubly commuting shift operators $S=(S_1,\ldots,S_n)$ on $\clh$, a bounded operator $T$ on $\clh$ is said be $S$-\textit{Toeplitz} if $S_k^*TS_k=T$ for all $k \in \{1,\ldots,n\}$.
\end{defn}
The condition $S_k^* T S_k = T$ for all $k \in \{1,\ldots,n\}$, imply that 
\[
M_{z_k}^* U T U^* M_{z_k} = U  S_k^* T S_k U^* = U T U^*,
\]
for all $k \in \{1,\ldots,n\}$. From Theorem \ref{Toep_char}, it follows that $UTU^*$ is a Toeplitz operator on $H_{\clw(S)}^2(\D^n)$ and hence, there exists a symbol $\Phi \in L_{\clb(\cle)}^{\infty}(\mathbb{T}^n)$ such that
\begin{equation}\label{d3}
UTU^* = T_{\Phi}.
\end{equation}
Our main result in this section is also related to $S$-analytic operators as introduced by Rosenblum and Rovnyak in \cite[Section 1.6, Page 6]{RR}.

\begin{defn}
Given a $n$-tuple of doubly commuting shift operators $S=(S_1,\ldots,S_n)$ on a Hilbert space $\clh$, a bounded operator $T$ on $\clh$ is said to be $S$-analytic if 
\[
S_i T = T S_i \quad (\forall i \in \{1,\ldots,n\}).
\]
\end{defn}
We are now ready to characterize $S$-Toeplitz operators on $\clh$ that are partially isometric.

\begin{thm}
Let $S=(S_1,\ldots,S_n)$ be a $n$-tuple of doubly commuting shift operators on $\clh$ and let $T$ be a $S$-Toeplitz operator on $\clh$. Then the following are equivalent\\
$(i)$~$T$ is a partial isometry,\\
$(ii)$~there exists partially isometric $S$-analytic operators $W_1, W_2$ with equal final spaces such that
\[
T = W_1 W_2^*,
\]
and the pair $(W_1,W_2)$ satisfy the following equivalent conditions for all $k \in \{1,\ldots, n\}$:\\
$(a)$~$S_k^* W_1 P_{\ker S_k^*} W_2^* S_k = 0$, \\
$(b)$~$S_k^* W_1  P_{\clw(S)} W_2^* S_k = 0$.
\end{thm}
\begin{proof}
Since $T$ is a partial isometry, the operator $UTU^*$ is also a partial isometry, and hence, using condition (\ref{d3}), we get $T_{\Phi}$ is also a partial isometry. Now, from Theorem \ref{mainthm1}, it follows that there must exist a Hilbert space $\clf$ and inner functions $\Gamma, \Psi \in H_{\clb(\clf,\cle)}^{\infty}(\D^n)$ such that
\[
T_{\Phi} = M_{\Gamma}M_{\Psi}^*.
\]
Since $\Gamma$ is an inner function $\Gamma(e^{it}):\clf \raro \cle$ is an isometry a.e. on $\mathbb{T}^n$. We can choose any such isometry $\Gamma(e^{it})$ and extend it to an isometry $I_{H^2(\D^n)} \otimes \Gamma(e^{it}): H_{\clf}^2(\D^n) \raro H_{\cle}^2(\D^n)$. For the sake of computation, we will denote this isometry simply as $i$. This construction implies that
\begin{equation}\label{inter2}
M_{z_k} i = i M_{z_k}; \quad M_{z_k}^* i = i M_{z_k},
\end{equation}
for all $k \in \{1,\ldots,n\}$.  Furthermore, using this isometry, we get 
\[
T = U^*T_{\Phi}U = U^* M_{\Gamma}M_{\Psi}^* U = U^* M_{\Gamma}i^* UU^* i M_{\Psi}^*U = W_1 W_2^*,
\]
where $W_1 = U^* M_{\Gamma}i^*U$ and $W_2 = U^* M_{\Psi}i^* U$.  Now, for any $k \in \{1,\ldots,n\}$, we get
\begin{align*}
W_1 S_k = U^* M_{\Gamma}i^*U S_k = U^* M_{\Gamma}i^* M_{z_k} U = U^* M_{z_k} M_{\Gamma}i^* U &= S_k U^* M_{\Gamma}i^* U\\
&= S_k W_1,
\end{align*}
and similarly,
\begin{align*}
W_2 S_k = U^* M_{\Psi}i^*U S_k = U^* M_{\Psi}i^* M_{z_k} U = U^* M_{z_k} M_{\Psi}i^* U &= S_k U^* M_{\Gamma}i^* U \\
&= S_k W_2.
\end{align*}
Thus, we see that both $W_1$ and $W_2$ are $S$-analytic operators on $\clh$. Furthermore,
\[
W_1^* W_1 = U^* i i^* U = W_2^* W_2,
\]
shows that both $W_1$, and $W_2$ are partial isometries with equal final spaces. Moreover, using Lemma \ref{toepcondn2}, we know that $M_{\Gamma}M_{\Psi}^*$ is a Toeplitz operator if and only if it satisfies the following equivalent conditions for all $k \in \{1,\ldots,n\}$.
\[
M_{z_k}^* M_{\Gamma}P_{\ker M_{z_k}^*} M_{\Psi}^* M_{z_k} = 0; \quad M_{z_k}^* M_{\Gamma}P_{\cle} M_{\Psi}^* M_{z_k} = 0.
\]
Now, the above conditions are again equivalent to 
\begin{equation}\label{d1}
U^* M_{z_k}^* M_{\Gamma} i^* U U^* i P_{\ker M_{z_k}^*} M_{\Psi}^* M_{z_k} U= 0,
\end{equation}
\begin{equation}\label{d2}
U^* M_{z_k}^* M_{\Gamma} i^*U U^*iP_{\cle}M_{\Psi}^* M_{z_k} U= 0,
\end{equation}
respectively. Using the intertwining relations in (\ref{inter2}), the first condition (\ref{d1}), turns into
\[
U^* M_{z_k}^* M_{\Gamma} i^* U U^* P_{\ker M_{z_k}^*} i M_{\Psi}^* M_{z_k} U= 0.
\]
Using the intertwining relations in (\ref{inter}), we get
\[
S_k^* U^* M_{\Gamma} i^* U P_{\ker S_k^*} U^* i M_{\Psi}^*  U S_k= 0, 
\]
that is, $S_k^* W_1 P_{\ker S_k^*} W_2^* S_k= 0$. Similarly, condition (\ref{d2}) is equivalent to
\[
U^* M_{z_k}^* M_{\Gamma} i^*U U^* P_{\cle}i M_{\Psi}^* M_{z_i} U= 0,
\]
and therefore, using the intertwining relations in (\ref{inter}), we get
\[
S_k^* U^* M_{\Gamma} i^*U P_{\clw(S)} U^* i M_{\Psi}^* U S_k= 0,
\]
Thus, this condition is again equivalent to
\[
S_k^* W_1 P_{\clw(S)} W_2^* S_k = 0.
\]
This completes the proof.
\end{proof}
In the above characterization, it is indeed surprising how the product $T = W_1 W_2^*$ can be a partial isometry without any one of the $W_1, W_2$ being an isometry. It is because both the partial isometries have equal final spaces. In particular, this implies that
\[
W_1 W_2^* W_2 W_1^* = W_1 W_1^* W_1 W_1^* = W_1 W_1^*.
\]
Hence, $W_1 W_2^*$ must be a partial isometry.
\section{An alternative proof for Toeplitz operators on $H^2(\D^n)$}\label{scalar}

This section will adapt the preceding results' ideas to characterize partially isometric Toeplitz operators on the scalar-valued Hardy space on the unit polydisc. This is an alternative approach to recently obtained results by  Deepak--Pradhan--Sarkar in \cite{KPS}.

\begin{thm}\label{comcondn}
Let $\zeta, \psi$ be bounded analytic functions $\D^n$.Then the following are equivalent\\
$(i)$~$M_{\zeta}M_{\psi}^*$ is a Toeplitz operator,\\
$(ii)$~$M_{\zeta}$ and $M_{\psi}^*$ depend on disjoint set of variables.\\
$(iii)$~$[M_{\zeta},M_{\psi}^*]=0$.
\end{thm}
\begin{proof}
Let us denote the bounded analytic functions $\zeta, \psi$ on $\D^n$ in the following form
\[
\zeta(\bm{z}):= \sum_{\bm{l} \in \Nat^n} a_{\bm{l}} z^{\bm{l}}  ; \quad \psi(\bm{z}):= \sum_{\bm{m} \in \Nat^n} b_{\bm{m}} z^{\bm{m}} \quad (z \in \D^n).
\]
From Theorem \ref{main2}, we know that $M_{\zeta}M_{\psi}^*$ is a Toeplitz operator if and only if,
\[
a_{\bm{l}+ e_i} \bar{b}_{\bm{m}+ e_i} = 0.
\]
for all $\bm{l}, \bm{m} \in \Nat^n$ and $i \in \{1,\ldots,n\}$. From this identity it follows that for any given $i \in \{1,\ldots,n\}$, either $a_{\bm{l}+ e_i} =0$ for all $\bm{l} \in \Nat^n$, or $b_{\bm{m}+ e_i} = 0$ for all $\bm{m} \in \Nat^n$ or both.  If $a_{\bm{l}+ e_i} =0$, for all $\bm{l} \in \Nat^n$, then it implies that $\zeta(\bm{z})$ does not depend on the $z_i$-variable. Similarly, if $b_{\bm{m}+ e_i} =0$, for all $\bm{m} \in \Nat^n$, then it will imply that $\psi(\bm{z})$ does not depend on the $z_i$-variable. Thus, from this discussion it follows that if $M_{\zeta}M_{\psi}^*$ is a Toeplitz operator then for each $i \in \{1,\ldots,n\}$,  any one of the following cases can hold:\\
$(i)$~$\zeta(\bm{z})$ depend on the variable $z_i$, but $\psi(\bm{z})$ does not depend on the variable $z_i$,\\
$(ii)$~$\psi(\bm{z})$ depend on the variable $z_i$, but $\zeta(\bm{z})$ does not depend on the variable $z_i$,\\
$(iii)$~both $\zeta(\bm{z})$ and $\psi(\bm{z})$ does not depend on the variable $z_i$.\\
This completes the $(i) \implies (ii)$ proof.  The other directions $(ii) \implies (iii)$, and $(iii) \implies (i)$ can be easily verified.
\end{proof}
Let us end this section with a new proof for the main result in \cite[Theorem 1.1]{KPS}.
\begin{thm}
A Toeplitz operator $T_{\phi}$ on $H^2(\D^n)$ is a partial isometry if and only if there exist inner functions $\zeta, \psi \in H^{\infty}(\D^n)$ depending on the disjoint set of variables such that 
\[
T_{\phi} = M_{\psi}^* M_{\zeta}.
\]
\end{thm}
\begin{proof}
We proceed as in the proof of Theorem \ref{mainthm1}. Since $T_{\phi}$ is a partial isometry, we know from Theorem \ref{doubcom}, that both $\ran T_{\phi}$ and $\ran T_{\phi}^*$ are Beurling type invariant subspaces of $H^2(\D^n)$. Hence, there must exist inner functions $\gamma, \psi \in H^{\infty}(\D^n)$ such that 
\[
\ran T_{\phi} =M_{\gamma} H^2(\D^n); \quad \ran T_{\phi}^* = M_{\psi} H^2(\D^n),
\]
If we follow the proof of Theorem \ref{mainthm1}, we can construct an unitary $X: \mathbb{C} \raro \mathbb{C}$ such that 
\[
T_{\phi} = M_{\gamma} X M_{\psi}^*.
\]
In other words, $T_{\phi} = M_{\zeta} M_{\psi}^*$, where 
\[
\zeta(\bm{z}):= \lambda \gamma(\bm{z}) \quad (\bm{z} \in \D^n),
\]
for some uni-modular constant $\lambda \in \mathbb{C}$.  From the above Theorem \ref{comcondn}, $M_{\zeta} M_{\psi}^*$ is a Toeplitz operator if and only if $\zeta$ and $\psi$ depend on the disjoint set of variables. Thus, we can write
\[
T_{\phi} = M_{\psi}^* M_{\zeta}.
\]
Conversely, if the Toeplitz operator admits the factorization $T_{\phi} = M_{\psi}^* M_{\zeta}$ for some inner functions $\psi$ and $\zeta$ depending on disjoint set of variables, then
\[
T_{\phi}T_{\phi}^* = M_{\psi}^* M_{\zeta} M_{\zeta}^* M_{\psi} = M_{\psi}^* M_{\psi}  M_{\zeta} M_{\zeta}^* = M_{\zeta} M_{\zeta}^*,
\]
shows that $T_{\phi}$ is a partial isometry. This completes the proof.
\end{proof}
Our methods, as shown in the above result, can be used similarly to give an alternative proof of the following characterization of Brown and Douglas \cite{BD}.
\begin{thm}
A Toeplitz operator $T_{\phi}$ on $H^2(\D)$ is a partial isometry if and only if there exists a inner function $\theta \in H^{\infty}(\D)$ such that $T_{\phi} = M_{\theta}  \text{ or } T_{\phi} = M_{\theta}^*.$
\end{thm}
\section{Some related questions}\label{questions}
Based on the results of this article, we can ask several interesting questions for Toeplitz operators on vector-valued Hardy spaces worthy of further investigation. \vspace{1mm}\\
(\textbf{I}) Characterize Toeplitz operators with shift-invariant range spaces. \vspace{1mm}\\
A significant contribution of this article is to show that if we start with a somewhat nice Toeplitz operator, then the range is Beurling-type. This gives rise to a fascinating problem. \vspace{1mm}\\
(\textbf{II}) Characterize Toeplitz operators with Beurling-type range spaces.\vspace{1mm}\\
In Corollary \ref{symb}, we have seen that a partially isometric Toeplitz operator has a partially isometric symbol, a.e. on $\mathbb{T}^n$, but what can we say about the converse? \vspace{1mm}\\
(\textbf{III}) Characterize symbols $\Phi \in L_{\clb(\cle)}^{\infty}(\mathbb{T}^n)$ such that $T_{\Phi}$ is a partial isometry.\vspace{1mm}\\
Unlike in the scalar cases, it is still not clear when partially isometric Toeplitz operators on $H_{\cle}^2(\D^n)$ become power partial isometries, so we end with the following question. \vspace{1mm}\\
(\textbf{IV}) Characterize Toeplitz operators, which are power partial isometries. 

\section*{Acknowledgement}
The author acknowledges the hospitality extended by the Indian Institute of Science, where a part of this work was completed. The Department of Science and Technology supports the author via the INSPIRE Faculty research grant DST/INSPIRE/04/2019/000769.


\begin{thebibliography}{99}

\bibitem{BRG}
J.A. Ball, I. Gohberg, L. Rodman, 
\emph{Interpolation of rational matrix functions,}
Oper. Theory Adv. Appl., 45
Birkhäuser Verlag, Basel, 1990. xii+605 pp.
ISBN:3-7643-2476-7.

\bibitem{BD}
A.  Brown and R.G. Douglas,  
\emph{Partially isometric Toeplitz operators,} Proc. Amer. Math. Soc. 16 (1965), 681--682. 

\bibitem{BH} 
A. Brown, P.R. Halmos, 
\emph{Algebraic properties of Toeplitz operators,}
Journal für die reine und angewandte Mathematik (1964), Volume: 213, page 89--102

\bibitem{BS}
A. B\"ottcher, B. Silbermann, 
\emph{Analysis of Toeplitz operators,}Springer-Verlag, Berlin, 1990. 512 pp. ISBN: 3-540-52147-X.

\bibitem{Cowen}
C. C. Cowen, 
\emph{Hyponormality of Toeplitz operators,} Proc. Amer. Math. Soc. 103:3
(1988), 809–812

\bibitem{CHL4}
R. E. Curto, I.S. Hwang, W.Y. Lee,
\emph{Hyponormality and subnormality of block Toeplitz operators},
Advances in Mathematics, Volume 230, Issues 4--6, 2012, Pages 2094--2151.

\bibitem{CHL}
R.E. Curto, I.S. Hwang, W.Y. Lee,
\emph{Which subnormal Toeplitz operators are normal or analytic?} J. Funct. Anal. 263 (2012), no. 8, 2333--2354. 

\bibitem{CHL2}
R.E. Curto, I.S. Hwang, W.Y. Lee,
\emph{Matrix functions of bounded type: an interplay between function theory and operator theory,} Mem. Amer. Math. Soc. 260 (2019), no. 1253, v+100 pp.

\bibitem{CHL3}
R.E. Curto, I.S. Hwang, W.Y. Lee, \emph{Operator-valued rational functions.,} J. Funct. Anal. 283 (2022), no. 9, Paper No. 109640, 23 pp.

\bibitem{DPS}
R. Debnath, D. K. Pradhan, J. Sarkar,
\emph{Pairs of inner projections and two applications,} Journal of Functional Analysis,
Volume 286, Issue 2, 2024.

\bibitem{KPS}
Deepak K. D., D. Pradhan and J. Sarkar, 
\emph{Partially isometric Toeplitz operators on the polydisc,} Bulletin of the London Mathematical Society, Volume 54, (2022), 1350--1362.

\bibitem{Douglas2}
R.G. Douglas, 
\emph{On majorization, factorization, and range inclusion of operators on Hilbert space,}
Proc. Amer. Math. Soc. 17 (1966), 413--415. 

\bibitem{Douglas}
R.G. Douglas, 
\emph{ Banach algebra techniques in the theory of Toeplitz operators,}  CBMS Regional Conference Series in Mathematics, No. 15. American Mathematical Society, Providence, RI, 1973. v+53 pp. 

\bibitem{FF}
C. Foias, A.E. Frazho, 
\emph{The commutant lifting approach to interpolation problems,}
Oper. Theory Adv. Appl., 44
Birkhäuser Verlag, Basel, 1990. xxiv+632 pp.

\bibitem{GZ}
C. Gu, D. Zheng,
\emph{Products of block Toeplitz operators,} Pacific Journal of Mathematics, 
Vol. 185 (1998), No. 1, 115--148.

 \bibitem{Gu2}
C. Gu, 
\emph{Some algebraic properties of Toeplitz
and Hankel operators on polydisk,}
Arch. Math. 80 (2003) 393--405.

\bibitem{GHR}
C. Gu, J. Hendricks, and D. Rutherford, \emph{Hyponormality of block Toeplitz operators,} Pacific J. Math. 223 (2006), 95--111.

\bibitem{Gu}
C. Gu, 
\emph{Multiplicative properties of infinite block Toeplitz and Hankel matrices,} Toeplitz operators and random matrices—in memory of Harold Widom, 371--399,
Oper. Theory Adv. Appl., 289, Birkhäuser/Springer, Cham, [2022].

\bibitem{GL}
C. Gu, S. Luo,
\emph{Invariant subspaces of the direct sum of forward and backward shifts on vector-valued Hardy spaces,} J. Funct. Anal. 282 (2022), no.9, Paper No. 109419, 31 pp.

\bibitem{Halmos}
P. R. Halmos,
\emph{Shifts on Hilbert spaces
Halmos,} J. Reine Angew. Math. 208 (1961), 102--112.

\bibitem{HM}
P. R. Halmos and J. McLaughlin, 
\emph{Partial isometries,} Pacific J. Math. 13 (1963), 585--596.

\bibitem{HW}
P. R. Halmos and L. Wallen. 
\emph{Powers of partial isometries,} Indiana Univ. Math. J. 19 (1970), 657--663.

\bibitem{Hayashi}
E. Hayashi,
\emph{The kernel of a Toeplitz operator,}
Integral Equations Operator Theory 9 (1986),
588--591.

\bibitem{MSS}
A. Maji, J. Sarkar and S. Sarkar, 
\emph{Toeplitz and asymptotic Toeplitz operators on $H^2(\D^n)$,} Bull. Sci. Math. 146 (2018), 33--49.

\bibitem{Mandrekar}
V. Mandrekar, 
\emph{The validity of Beurling theorems in polydiscs,} Proc. Amer. Math. Soc. 103 (1988), 145--148.

\bibitem{MP}
M. Martin and M. Putinar,
\emph{Lectures on hyponormal operators,}
Operator Theory: Advances and Applications, 39. Birkhäuser Verlag, Basel, 1989. 304 pp. ISBN: 3-7643-2329-9.

\bibitem{NF}
B. Sz.-Nagy, C. Foias, H. Bercovici, L. K\'erchy,
\emph{Harmonic analysis of operators on Hilbert space,} Second edition. Revised and enlarged edition. Universitext. Springer, New York, 2010. xiv+474 pp. ISBN: 978-1-4419-6093-1.

\bibitem{Page}
L.B. Page, 
\emph{Bounded and compact vectorial Hankel operators,} Trans. Amer. Math. Soc.150 (1970), 529--539.

\bibitem{Peller}
V.V. Peller, 
\emph{Hankel Operators and Their Applications,} Springer-Verlag, 2003.

\bibitem{RR}
M. Rosenblum and J. Rovnyak, 
\emph{Hardy classes and operator theory. (English summary).} Corrected reprint of the 1985 original Dover Publications, Inc., Mineola, NY, 1997. xiv+161 pp. ISBN:0-486-69536-0.

\bibitem{Sarason}
D. Sarason
\emph{Kernels of Toeplitz operators,} Toeplitz Operators and Related Topics, Birkh¨auser
Verlag, Basel, 1994, pp. 153--164.

\bibitem{SSW}
J. Sarkar, A. Sasane and B. Wick, 
\emph{Doubly commuting submodules of the Hardy module over polydiscs,} Studia Mathematica, 217 (2013), no 2, 179--192.

\bibitem{JS} 
J. Sarkar,
\emph{Wold decomposition for doubly commuting isometries,} Linear Algebra and its Applications, 445 (2014), 289--301. 

\bibitem{Timotin}
D. Timotin, 
\emph{ The invariant subspaces of $S \oplus S^*$,} Concr. Oper. 7 (2020), no. 1, 116--123. 

\end{thebibliography}
\end{document}